\documentclass[12pt]{article}
\usepackage{mathrsfs}
\usepackage{amsmath, amssymb, amsthm, latexsym, amsfonts, float, epsfig, epstopdf, subcaption, caption, color,graphicx}

\begin{document}

\numberwithin{equation}{section}

\newcommand{\s}{\sigma}
\renewcommand{\k}{\kappa}
\newcommand{\p}{\partial}
\newcommand{\D}{\Delta}
\newcommand{\om}{\omega}
\newcommand{\Om}{\Omega}
\renewcommand{\phi}{\varphi}
\newcommand{\e}{\epsilon}
\renewcommand{\a}{\alpha}
\renewcommand{\b}{\beta}
\newcommand{\N}{{\mathbb N}}
\newcommand{\R}{{\mathbb R}}
   \newcommand{\eps}{\varepsilon}
   \newcommand{\EX}{{\Bbb{E}}}
   \newcommand{\PX}{{\Bbb{P}}}

\newcommand{\cF}{{\cal F}}
\newcommand{\cG}{{\cal G}}
\newcommand{\cD}{{\cal D}}
\newcommand{\cO}{{\cal O}}

\newcommand{\de}{\delta}

\newcommand{\grad}{\nabla}
\newcommand{\n}{\nabla}
\newcommand{\curl}{\nabla \times}
\newcommand{\dive}{\nabla \cdot}

\newcommand{\ddt}{\frac{d}{dt}}
\newcommand{\la}{{\lambda}}

\newtheorem{theorem}{Theorem}
\newtheorem{lemma}{Lemma}
\newtheorem{remark}{Remark}
\newtheorem{example}{Example}
\newtheorem{note}{Note}
\newtheorem{definition}{Definition}
\newtheorem{corollary}{Corollary}

\def\proof{\mbox {\it Proof.~}}

\makeatletter\def\theequation{\arabic{section}.\arabic{equation}}\makeatother

\centerline{\large \bf Slow manifolds for stochastic systems}
\centerline{\large \bf with non-Gaussian stable L\'evy noise\footnote{This work was partly supported
by the NSFC grants 11301197, 11301403, 11371367, 11271290 and 0118011074.}\hspace{2mm}
\vspace{1mm}\vspace{1mm}\\ }
\bigskip
\centerline{\bf Shenglan Yuan$^{a,b,c,}\footnote{shenglanyuan@hust.edu.cn}$, Jianyu Hu$^{a,b,c,}\footnote{hujianyu@hust.edu.cn}$,
Xianming Liu$^{b,c,}\footnote{xmliu@hust.edu.cn}$,Jinqiao Duan$^{d,}\footnote{ duan@iit.edu}$}
\smallskip
\centerline{${}^a$Center for Mathematical Sciences,}
 \centerline{${}^b$School of Mathematics and Statistics,}
  \centerline{${}^c$Hubei Key Laboratory of Engineering Modeling and Scientific Computing,}
 \centerline{Huazhong University of Sciences and Technology, Wuhan 430074, China}
\centerline{${}^d$Department of Applied Mathematics}
\centerline{Illinois Institute of Technology, Chicago, IL 60616, USA}
\bigskip\par

\begin{abstract}
This work is concerned with the dynamics of a class of slow-fast stochastic dynamical systems with non-Gaussian stable L\'evy noise with a scale parameter. Slow manifolds with exponentially tracking property are constructed, eliminating the fast variables to reduce the dimension of these coupled dynamical systems. It is shown that as the scale parameter tends to zero, the slow manifolds converge to critical manifolds in distribution, which helps understand long time dynamics. The approximation of slow manifolds with error estimate in distribution are also considered.

\end{abstract}


\setcounter{secnumdepth}{5} \setcounter{tocdepth}{5}

\makeatletter
    \newcommand\figcaption{\def\@captype{figure}\caption}
    \newcommand\tabcaption{\def\@captype{table}\caption}
\makeatother


\section{\bf Introduction }

Stochastic effects are ubiquitous in complex systems in science and engineering \cite{IM02,PZ07,Wo01}. Although random mechanisms may appear to be very small or very fast, their long time impacts on the system evolution may be delicate or even profound, which has been observed in, for example, stochastic bifurcation, stochastic optimal control, stochastic resonance and noise-induced pattern formation \cite{Ar03,PK08,Wa02}. Mathematical modeling of complex systems under uncertainty often leads to stochastic differential equations (SDEs) \cite{Ku04,Pr04,Ro05}. Fluctuations appeared in the SDEs are often non-Gaussian (e.g., L\'evy motion) rather than Gaussian (e.g., Brownian motion); see Schilling \cite{BSW14,SP12}.

We consider the slow-fast stochastic dynamical system where the fast dynamic is driven by $\alpha$-stable noise, see \cite{GL13,POC99}.
In particular, we study
\begin{eqnarray}\label{3.2}
\left\{\begin{array}{l}
 dx^\varepsilon=Sx^\varepsilon dt+g_{1}(x^\varepsilon,y^\varepsilon)dt,~  \\
dy^\varepsilon=\frac{1}{\varepsilon}Fy^\varepsilon dt+\frac{1}{\varepsilon}
 g_{2}(x^\varepsilon,y^\varepsilon)dt+\sigma\varepsilon^{-\frac{1}{\alpha}}dL_t^{\alpha},
 \end{array}\right.
\end{eqnarray}
where $(x^\varepsilon,y^\varepsilon)\in\mathbb{R}^{n_{1}}\times \mathbb{R}^{n_{2}}$, $\varepsilon$ is a small positive parameter measuring slow and fast time scale separation such that in a formal sense
\begin{equation*}
  |\frac{dx^\varepsilon}{dt}|\ll|\frac{dy^\varepsilon}{dt}|,
\end{equation*}
where we denote by $|.|$ the Euclidean norm.
The $n_{1}\times n_{1}$ matrix $S$ with all eigenvalues with non-negative real part, $F$ is a $n_{2}\times n_{2}$
matrix whose eigenvalues have negative real part.
Nonlinearities $g_{i}:\mathbb{R}^{n_{1}}\times\mathbb{R}^{n_{2}}\mapsto\mathbb{R}^{n_{i}}, i=1,2$
are Lipschitz continuous functions with $g_{i}(0,0)=0$. $\{L_t^{\alpha}: t\in\mathbb{R}\}$ is a two-sided $\mathbb{R}^{n_{2}}$-valued $\alpha$-stable
L\'evy process on a probability space $(\Omega,\mathcal{F},\mathbb{P})$, where $1<\alpha<2$ is the index of stability \cite{Ap04,Du15}. The strength of noise in the fast equation is chosen to be $\varepsilon^{-\frac{1}{\alpha}}$
to balance the stochastic force and deterministic force. $\sigma>0$ is the intensity of noise.

Invariant manifolds are geometric structures in state space that are useful in investigating the dynamical behaviors of stochastic systems; see \cite{CS10,CDZ14,DLS03,DLS04}. A slow manifold is a special invariant manifold of a slow-fast system, where the fast variable is represented by the slow variable
and the scale parameter $\varepsilon$ is small. Moreover, it exponentially attracts other orbits. A critical manifold of a slow-fast system is the slow manifold corresponding to the zero scale parameter \cite{Ge05}. The theory of slow manifolds and critical
manifolds provides us with a powerful tool for analyzing geometric structures of slow-fast stochastic dynamical systems, and reducing the dimension of those systems.

For a system like \eqref{3.2} based on Brownian noise ($\alpha=2$), the existence of the slow manifold and its approximation has been extensively studied \cite{Ba98,DW14,FLD12,SS08,WR13}. The dynamics of individual sample solution paths have also been quantified; see \cite{ BB06,KP13,WR12}. Moreover, Ren and Duan \cite{RDJ15,RDW15} provided a numerical simulation for the slow manifold and established its parameter estimation. The study of the dynamics generated by SDEs under non-Gaussian L\'evy noise is still in its infancy, but some interesting works are emerging \cite{Du15,Ku16,LAD11}.

The main goal of this paper it to investigate the slow manifold of dynamical system \eqref{3.2} driven by $\alpha$-stable
L\'evy process with $\alpha\in(1,2)$ in finite dimensional setting, and examine its approximation and structure.

We first introduce a random transformation based on the generalised Ornstein-Uhlenbeck process, such that a solution the system of SDEs \eqref{3.2} with $\alpha$-stable L\'evy noise can be represented as a transformed solution of random differential equations (RDEs) and vice versa. Then we prove that, for $0<\varepsilon\ll1$, the slow manifold $\mathcal{M}^{\varepsilon}$ with an exponential tracking property can be constructed as fixed point of the projected RDEs by using the Lyapunov-Perron method \cite{Bo89}. Thus as a consequence, with the inverse conversion, we can obtain the slow manifold $M^{\varepsilon}$ for the original SDE system. Subsequently we convert the above RDEs to new RDEs by taking the time scaling $t\rightarrow\varepsilon t$.  After that we use the Lyapunov-Perron method once again to establish the existence of the slow manifold $\mathcal{\tilde{M}}^{\varepsilon}$ for new RDE system, and denote $\mathcal{M}^{0}(\omega)$ as the critical manifold with zero scale parameter in particular. In addition, we show that $\mathcal{M}^{\varepsilon}$ is same as $\mathcal{\tilde{M}}^{\varepsilon}$ in distribution, and the distribution of $\mathcal{M}^{\varepsilon}(\omega)$ converges to the distribution of $\mathcal{M}^{0}(\omega)$, as $\varepsilon$ tends to zero.
Finally, we derive an asymptotic approximation for the slow manifold $\mathcal{M}^{\varepsilon}$ in distribution. Moreover, as part of ongoing studies, we try to study mean residence time on slow manifold, and generalise these results to consider system \eqref{3.2} in Hilbert spaces to study  infinite dimensional dynamics.

This paper is organized as follows. In Section 2, we recall some basic concepts in random dynamical systems, and construct metric dynamical systems driven by L\'evy processes with two-sided time. In Section 3, we recall random invariant manifolds and introduce hypotheses for the slow-fast system.
In Section 4, we show the existence of slow manifold (Theorem \ref{slowmanifold}), and measure the rate of slow manifold attract other dynamical orbits (Theorem \ref{Etp}). In Section 5, we prove that as the scale parameter tends to zero, the slow manifold converges to the critical manifold in distribution (Theorem \ref{theorem5}). In Section 6, we present numerical results using examples from mathematical biology to corroborate our analytical results.

\section{\bf Random Dynamical Systems}\label{s2}
We are going to introduce the main tools we need
to find inertial manifolds for systems of stochastic differential equations driven
by $\alpha$-stable L\'evy noise. These tools stem from the theory of random
dynamical systems; see Arnold \cite{Ar03}.

An appropriate model for noise is a metric dynamical system $\Theta=(\Omega,\mathcal{F},\mathbb{P},\theta)$, which consists of
a probability space $(\Omega,\mathcal{F},\mathbb{P})$ and a flow $\theta=\{\theta_{t}\}_{t\in\mathbb{R}}$:
\[
   \theta:\ \mathbb{R}\times\Omega\to\Omega,\ (t,\omega)\mapsto\theta_{t}\omega;\
   \theta_0={\rm id}_\Omega,\ \theta_{s+t}=\theta_s\circ\theta_t=:\theta_s\theta_t.
   \]
The flow $\theta$ is jointly $\mathcal{B}(\mathbb{R})\otimes\mathcal{F}-\mathcal{F}$ measurable.
All $\theta_t$ are measurably invertible with $\theta_t^{-1}=\theta_{-t}$. In addition, the probability measure $\mathbb{P}$ is
invariant (ergodic) with respect to the mappings $\{\theta_{t}\}_{t\in\mathbb{R}}$.

 For example, the L\'evy process with two side time represents a metric dynamical system. Let $L=(L_{t})_{t\geq0}$ with $L_{0}=0$ a.s. be a L\'evy process with values in $\mathbb{R}^{n}$ defined on the canonical probability space $(E^{\mathbb{R}^{+}},\mathcal{E}^{\mathbb{R}^{+}},\mathbb{P}_{\mathbb{R}^{+}})$, in which $E=\mathbb{R}^{n}$ endowed with the Borel $\sigma$-algebra $\mathcal{E}$. We can construct the corresponding two-sided L\'evy process $L_{t}(\omega):=\omega(t), t\in\mathbb{R}$ defined on $(E^{\mathbb{R}},\mathcal{E}^{\mathbb{R}},\mathbb{P}_{\mathbb{R}})$, see K\"{u}mmel \cite[p29]{Ku16}. Since the paths of a L\'evy process are c\`{a}dl\`{a}g; see \cite[Theorem 2.1.8]{Ap04}. We can define two-sided L\'evy process on the space
$(\Omega,\mathcal{F},\mathbb{P})$ instead of $(E^{\mathbb{R}},\mathcal{E}^{\mathbb{R}},\mathbb{P}_{\mathbb{R}})$, where $\Omega=\mathcal{D}_{0}(\mathbb{R},\mathbb{R}^{n})$ is the space of c\`{a}dl\`{a}g functions starting at $0$ given by
\begin{equation*}
  \mathcal{D}_0=\{\omega :\ \text{for } \forall t\in\mathbb{R},\ \lim_{s\uparrow t}\omega(s)=\omega(t-),\  \lim_{s\downarrow t}\omega(s)=\omega(t)\ \text{exist}\ \text{and} \ \omega(0)=0  \}.
\end{equation*}
This space equipped with Skorokhod's $\mathcal{J}_{1}$-topology generated by the metric $\rm{d}_{\mathbb{R}}$ is a Polish space. For functions $\omega_{1}, \omega_{2}\in\mathcal{D}_{0}$, ${\rm d}_{\mathbb{R}}(\omega_{1}, \omega_{2})$ is given by
\begin{eqnarray*}
  {\rm d}_{\mathbb{R}}(\omega_{1},\omega_{2})=\inf\left\{\varepsilon>0:
     |\omega_{1}(t)-\omega_{2}(\lambda t)|\leq\varepsilon,\ \big|\ln\frac{\arctan(\lambda t)-\arctan(\lambda s)}{\arctan(t)-\arctan(s)}\big|\leq\varepsilon\right.
     \\
    \left.\text{for every}~\ t, s\in\mathbb{R}\ \text{and some}\ \lambda\in\Lambda^{\mathbb{R}}\right\},
\end{eqnarray*}
where
\begin{equation*}
\Lambda^{\mathbb{R}}=\{\lambda :\mathbb{R}\rightarrow\mathbb{R};\ \lambda\ \text{is injective increasing},\ \lim\limits_{t\rightarrow-\infty}\lambda(t)=-\infty,\ \lim\limits_{t\rightarrow\infty}\lambda(t)=\infty \}.
\end{equation*}
$\mathcal{F}$ is the associated Borel $\sigma$-algebra $\mathcal{B}(\mathcal{D}_{0})=\mathcal{E}^{\mathbb{R}}\cap\mathcal{D}_{0}$, and $(\mathcal{D}_{0},\mathcal{B}(\mathcal{D}_{0}))$ is a separable metric space.
The probability measure $\mathbb{P}$ generated by $\mathbb{P}(\mathcal{D}_{0}\cap A):=\mathbb{P}_{\mathbb{R}}(A)$ for each $A\in\mathcal{E}^{\mathbb{R}}$. The flow $\theta$ is given by
\[
  \theta:\mathbb{R}\times\mathcal{D}_{0}\rightarrow\mathcal{D}_{0},\ \theta_{t}\omega(\cdot)\mapsto\omega(\cdot+t)-\omega(t),
\]
which is a Carath\'eodory function. It follows that $\theta$ is jointly measurable. Moreover, $\theta$
satisfies $\theta_{s}(\theta_{t}\omega(\cdot))=\omega(\cdot+s+t)-\omega(s)-(\omega(s+t)-\omega(s))=\theta_{s+t}\omega(\cdot)$ and $\theta_{0}\omega(\cdot)=\omega(\cdot)$. Thus $(\mathcal{D}_{0},\mathcal{B}(\mathcal{D}_{0}),\mathbb{P},(\theta_{t})_{t\in\mathbb{R}})$ is a metric dynamical system generated by L\'evy process with two-side time. Note that the probability measure $\mathbb{P}$ is ergodic with respect to the flow $\theta=(\theta_{t})_{t\in\mathbb{R}}$.

In the above we define metric dynamical system first, which will step in  $\theta$ in the complete definition of random dynamical system, that is strongly motivated by the measuability property combined with the cocycle property.

A random dynamical system taking values in the measurable space $(H,\mathcal{B}(H))$ over a  metric dynamical system $(\Omega,\mathcal{F},\mathbb{P},(\theta_{t})_{t\in\mathbb{R}})$ with time space $\mathbb{R}$ is given by a mapping
\begin{equation*}
    \phi: \mathbb{R}\times\Omega\times H \to H,
\end{equation*}
 that is jointly $\mathcal{B}(\mathbb{R})\otimes\mathcal{F}\otimes\mathcal{B}(H)-\mathcal{B}(H)$ measurable and satisfies the cocycle property:
\begin{equation}\label{cocycle}\begin{array}{c}
     \phi(0,\omega,\cdot)={\rm id}_{H}=\text{ identity on}\ H\ \text{ for each}\ \omega\in\Omega;\\
   \phi(t+s,\omega,\cdot)=\phi\big(t,\theta_s\omega,\phi(s,\omega,\cdot)\big)\ \text{ for each}\ s,\,t\in\mathbb{R},\ \omega\in\Omega.\end{array}
\end{equation}
For our application, in the sequel we suppose $H=\mathbb{R}^{n}=\mathbb{R}^{n_{1}}\times \mathbb{R}^{n_{2}}$.

Note that if $\phi$ satisfies the cocycle property \eqref{cocycle} for almost all $\omega\in\Omega\backslash\mathcal{N}_{s,t}$ (where the exceptional set $\mathcal{N}_{s,t}$ can depend on $s,t$), then we say $\phi$ forms a crude cocycle instead of a perfect cocycle. In this case, to get a random dynamical system, sometimes we can do a perfection of the crude cocycle, such that the cocycle property is valid for each and every $\omega\in\Omega$, see Scheutzow $\cite{Sc96}$.

We now recall some objects to help understand the dynamics of a random
dynamical system.

A random variable $\omega\mapsto X(\omega)$ with values in $H$ is called a stationary orbit (or random
fixed point) for a random dynamical system $\phi$ if
\begin{equation}\nonumber
     \phi(t,\omega,X(\omega))=X(\theta_{t}\omega)\ \text{ for }\ t\in\mathbb{R},\ \omega\in\Omega.
\end{equation}
Since the probability measure $\mathbb{P}$ is
invariant with respect to $\{\theta_{t}\}_{t\in\mathbb{R}}$, the random variables $\omega\mapsto X(\theta_{t}\omega)$
have the same distribution as $\omega\mapsto X(\omega)$. Thus $(t,\omega)\mapsto X(\theta_{t}\omega)$ is a stationary process , and therefore a stationary solution to the stochastic differential equation generating the random dynamical system $\phi$.

A family of nonempty closed sets $M=\{M(\omega)\subset\mathbb{R}^{n},\omega\in\Omega\}$
is call a random set for a random dynamical system $\varphi$, if the mapping
\[
\omega\mapsto{\rm dist}(z,M(\omega)):=\inf\limits_{z^{'}\in M(\omega)}|z-z^{'}|
\]
is a random variable for every $z\in\mathbb{R}^{n}$. Moreover $M$ is called an (positively) invariant set, if
\begin{equation}\label{invariant set}
\varphi(t,\omega,M(\omega))\subseteq M(\theta_t\omega)\ \text{ for }\ t\in\mathbb{R},\ \omega\in\Omega.
\end{equation}

Let
\[
h:\Omega\times\mathbb{R}^{n_{1}}\rightarrow\mathbb{R}^{n_{2}},\ (\omega,x)\mapsto h(\omega,x)
\]
be a function such that for all $\omega\in\Omega$, $x\mapsto h(\omega,x)$ is Lipschitz continuous, and for any $x\in\mathbb{R}^{n_{1}}$,  $\omega\mapsto h(\omega,x)$ is a random variable. We define
\[
\mathcal{M}(\omega)=\{(x,h(\omega,x)):x\in\mathbb{R}^{n_{1}}\}
\]
such that $\mathcal{M}=\{\mathcal{M}(\omega)\subset\mathbb{R}^{n},\omega\in\Omega\}$ can be represented as a graph of $h$.
It can be shown \cite[Lemma 2.1]{SS08} that $\mathcal{M}$ is a random set.

If $\mathcal{M}(\omega), \omega\in\Omega$ also satisfy \eqref{invariant set}, $\mathcal{M}$ is called a Lipschitz continuous invariant manifold.
Furthermore, $\mathcal{M}$ is said to have an exponential tracking property if for all $\omega\in\Omega$, there exists an $z^{'}\in\mathcal{M}(\omega)$ such that,
\[
|\varphi(t,\omega,z)-\varphi(t,\omega,z^{'})|\leq C(z,z^{'},\omega)e^{-ct}|z-z^{'}|,\ t\geq0,\ z\in \mathbb{R}^{n},
\]
where $C$ is a positive random variable depending on $z$ and $z^{'}$, while $c$ is a positive constant. Then $\mathcal{M}$ is called a random slow manifold with respect to the random dynamical system $\varphi$.

Let $\varphi$ and $\tilde{\varphi}$ be two random dynamical systems. Then $\varphi$ and $\tilde{\varphi}$ are called conjugated, if there is a random mapping $T:\Omega\times\mathbb{R}^{n}\rightarrow\mathbb{R}^{n}$, such that for all $\omega\in\Omega$, $(t,z)\mapsto T(\theta_{t}\omega,z)$ is a Carath\'eodory function, for every $t\in\mathbb{R}^{n}$ and $\omega\in\Omega$, $z\mapsto T(\theta_{t}\omega,z)$ is homeomorphic, and
\begin{equation}\label{conjugacy}
\tilde{\varphi}(t,\omega,z)=T(\theta_{t}\omega,\varphi(t,\omega,T^{-1}(\omega,z))),\ z\in\mathbb{R}^{n},
\end{equation}
where $T^{-1}:\Omega\times\mathbb{R}^{n}\rightarrow\mathbb{R}^{n}$ is the corresponding inverse mapping of $T$. Note that $T$ provides a random transformation form $\varphi$ to $\tilde{\varphi}$ that may be simpler to treat. If $\tilde{M}$ is a invariant set for the random dynamical system $\tilde{\varphi}$, we define
\[
M(\cdot):=T^{-1}(\cdot,\tilde{M}(.)).
\]
From the properties of $T$, $M$ is also invariant set with respect to $\varphi$.


\section{Slow-Fast Dynamical Systems}\label{section3}


\bigskip

The theory of invariant manifolds and slow manifolds of random dynamical system are essential for the study of the solution orbits, and we can use it to simplify dynamical systems by reducing an random dynamical system on a lower-dimensional manifold.

For the slow-fast system \eqref{3.2} described by stochastic differential equations
with $\alpha$-stable L\'evy noise, the state space for slow variables is $\mathbb{R}^{n_{1}}$,
the state space for fast variables is $\mathbb{R}^{n_{2}}$. To construct the slow manifolds of system \eqref{3.2},
we introduce the following hypotheses.

Concerning the linear part of \eqref{3.2}, we suppose

$(A_{1})$. There are constants $\gamma_{s}>0$ and $\gamma_{f}<0$ such that
\begin{align*}
|e^{St}x|&\leq e^{\gamma_{s}t}|x|,\ t\leq0,\ \mbox{for~all}~x\in\mathbb{R}^{n_{1}},\\
|e^{Ft}y|&\leq e^{\gamma_{f}t}|y|,\ t\geq0,\ \mbox{for~all}~y\in\mathbb{R}^{n_{2}}.
\end{align*}

With respect to the nonlinear parts of system \eqref{3.2}, we assume

$(A_{2})$. There exists a constant $K>0$
such that for all $(x_{i},y_{i})\in\mathbb{R}^{n},\ i=1,2,$,
\begin{eqnarray*}
  &&|g_{i}(x_{1},y_{1})-g_{i}(x_{2},y_{2})|\leq K(|x_{1}-x_{2}|+|y_{1}-y_{2}|),
\end{eqnarray*}
which implies that $g_{i}$ are continuous and thus measurable with respect to all variables. If $g_{i}$ is locally Lipshitz, but the corresponding deterministic system has a bounded absorbing set. By cutting off $g_{i}$ to zero outside a ball containing the absorbing set, the modified system has globally Lipschitz drift \cite{K13}.

For the proof of the existence of a random invariant manifold parametrized by $x\in\mathbb{R}^{n_{1}}$, we have to assume that the following spectral gap condition.

$(A_{3})$.  The decay rate $-\gamma_{f}$ of $e^{Ft}$ is larger than the Lipschitz constant $K$ of the
nonlinear parts in system \eqref{3.2}, i.e. $K<-\gamma_{f}$.
\begin{lemma}\label{lemma1}
Under hypothesis $(A_{1})$, the following linear stochastic differential equations
\begin{align}\label{3.3}
   d\eta^{\frac{1}{\varepsilon}}(t)&=\frac{1}{\varepsilon}F\eta^{\frac{1}{\varepsilon}}dt+\varepsilon^{-\frac{1}{\alpha}}dL_t^{\alpha},\ \eta^{\frac{1}{\varepsilon}}(0)=\eta_{0}^{\frac{1}{\varepsilon}},\\\label{3.4}
  d\xi(t)&=F\xi dt+dL_t^{\alpha},\ \xi(0)=\xi_{0},
\end{align}
have c\`{a}dl\`{a}g stationary solutions $\eta^{\frac{1}{\varepsilon}}(\theta_{t}\omega)$ and $\xi(\theta_{t}\omega)$
defined on
$\theta$-invariant set $\Omega$ of full measure, through the random variables
\begin{equation}\label{intial}
\eta^{\frac{1}{\varepsilon}}(\omega)=\varepsilon^{-\frac{1}{\alpha}}\int_{-\infty}^{0} e^{\frac{-Fs}{\varepsilon}}dL_{s}^{\alpha}(\omega),\ \xi(\omega)=\int_{-\infty}^{0}e^{-Fs}dL_s^{\alpha}(\omega),
\end{equation}
respectively. Moreover, they generate random dynamical systems.
\end{lemma}
\begin{proof}
The SDE \eqref{3.4} has unique c\`{a}dl\`{a}g solution
\begin{equation}\label{orbit}
\varphi(t,w,\xi_{0})=e^{Ft}\xi_{0}+\int_{0}^{t}e^{F(t-s)}dL_s^{\alpha}(\omega),
\end{equation}
for details see \cite{Ap04,Ku16,SY83}. It follows from \eqref{intial} and \eqref{orbit} that
\begin{align*}
\varphi(t,w,\xi(\omega))&=e^{Ft}\xi(\omega)+\int_{0}^{t}e^{F(t-s)}dL_s^{\alpha}(\omega)\\
       &=e^{Ft}\int_{-\infty}^{0}e^{-Fs}dL_s^{\alpha}(\omega)+\int_{0}^{t}e^{F(t-s)}dL_s^{\alpha}(\omega)\\
       &=\int_{-\infty}^{t}e^{F(t-s)}dL_s^{\alpha}(\omega).
\end{align*}
By \eqref{intial}, we also see that
\begin{align*}
\xi(\theta_{t}\omega)&=\int_{-\infty}^{0}e^{-Fs}dL_s^{\alpha}(\theta_t\omega)=\int_{-\infty}^{0}e^{-Fs}d\big(L_{t+s}^{\alpha}(\omega)-L_{t}^{\alpha}(\omega)\big)\\
 &=\int_{-\infty}^{0}e^{-Fs}dL_{t+s}^{\alpha}(\omega)=\int_{-\infty}^{t}e^{F(t-s)}dL_s^{\alpha}(\omega).    \end{align*}
Hence $\varphi(t,w,\xi(\omega))=\xi(\theta_{t}\omega)$ is a stationary orbit for \eqref{3.4}. Then we have
\begin{align*}
\xi(\theta_{t+s}\omega)&=\int_{-\infty}^{t+s}e^{F(t+s-r)}dL_r^{\alpha}(\omega)=e^{F(t+s)}\int_{-\infty}^{t}e^{-F(r+s)}dL_{r+s}^{\alpha}(\omega)\\
                 &=e^{F(t+s)}\int_{-\infty}^{t}e^{-F(r+s)}d\big(L_{r+s}^{\alpha}(\omega)-L_{s}^{\alpha}(\omega)\big)=\int_{-\infty}^{t}e^{F(t-r)}dL_r^{\alpha}(\theta_s\omega)\\
                 &=\xi(\theta_{t}\theta_{s}\omega),
\end{align*}
which implies $\xi$ generate a random dynamical system. Analogously
we obtain the SDE \eqref{3.3} whose unique solution is the
\emph{generalised Ornstein-Uhlenbeck process}
\begin{equation*}
\eta^{\frac{1}{\varepsilon}}(\theta_{t}\omega)=\varepsilon^{-\frac{1}{\alpha}}\int_{-\infty}^{t}e^{\frac{F(t-s)}{\varepsilon}}dL_{s}^{\alpha}.
\end{equation*}
\end{proof}
\begin{remark}
Since $\alpha$-stable process $L_{t}^{\alpha}$ satisfying $E\log(1+|L_{1}^{\alpha}|)<\infty$, $\eta^{\frac{1}{\varepsilon}}(\theta_{t}\omega)$ and $\xi(\theta_{t}\omega)$ are well-defined sationary semimartingales; see \cite[Remark 4.6]{Ku16}.
\end{remark}
\begin{lemma}\label{same distribution}
The process $\eta^{\frac{1}{\varepsilon}}(\theta_{t}\omega)$ has the same distribution
as the process $\xi(\theta_{\frac{t}{\varepsilon}}\omega)$, where $\eta^{\frac{1}{\varepsilon}}$ and
$\xi$ are defined in Lemma \ref{lemma1}.
\end{lemma}
\begin{proof}
From $\alpha$-stable process $L_{t}^{\alpha}$ are self-similar with Hurst index $\frac{1}{\alpha}$, i.e.,
\[
L_{ct}^{\alpha}\stackrel{d}{=}  c^{\frac{1}{\alpha}} L_{t}^{\alpha},
\]
where $``\stackrel{d}{=}``$ denotes equivalence (coincidence) in distribution, we have
\begin{align*}
\eta^{\frac{1}{\varepsilon}}(\theta_{t}\omega)&=\varepsilon^{-\frac{1}{\alpha}}\int_{-\infty}^{t} e^{\frac{F(t-s)}{\varepsilon}}dL_{s}^{\alpha}(\omega)=\int_{-\infty}^{\frac{t}{\varepsilon}}e^{F(\frac{t}{\varepsilon}-u)}\varepsilon^{-\frac{1}{\alpha}}dL_{\varepsilon u}^{\alpha}(\omega)\\
  &\stackrel{d}{=}\int_{-\infty}^{\frac{t}{\varepsilon}} e^{F(\frac{t}{\varepsilon}-u)}dL_{u}^{\alpha}(\omega)=\xi(\theta_{\frac{t}{\varepsilon}}\omega),
\end{align*}
which proves that $\eta^{\frac{1}{\varepsilon}}(\theta_{t}\omega)$ and $\xi(\theta_{\frac{t}{\varepsilon}}\omega)$ have the same distribution.
\end{proof}
\bigskip
Now we will transform the slow-fast stochastic dynamical system \eqref{3.2}
into a random dynamical system \cite{Ro08}. We introduce the random transformation
\begin{equation}\label{T}
\left(
  \begin{array}{ccc}
 \hat{ x}^{\varepsilon}\\
  \hat{y}^{\varepsilon}\\
  \end{array}
\right):=T^{\varepsilon}(\omega,x^{\varepsilon},y^{\varepsilon}):=
  \left(
  \begin{array}{ccc}
    &x^\varepsilon\\
    &y^\varepsilon-\sigma\eta^{\frac{1}{\varepsilon}}(\omega)\\
  \end{array}
\right).
\end{equation}
Then $(\hat{x}^\varepsilon(t), \hat{y}^\varepsilon(t))=T^{\varepsilon}(\theta_{t}\omega,x^{\varepsilon}(t),y^{\varepsilon}(t))$ satisfies
\begin{eqnarray}\label{3.6}
\left\{\begin{array}{l}
d\hat{x}^\varepsilon=S\hat{x}^\varepsilon dt+g_{1}(\hat{x}^\varepsilon,\hat{y}^\varepsilon+\sigma\eta^{\frac{1}{\varepsilon}}(\theta_{t}\omega))dt,\\
d\hat{y}^\varepsilon=\frac{1}{\varepsilon}F\hat{y}^\varepsilon dt+\frac{1}{\varepsilon}g_{2}(\hat{x}^\varepsilon,\hat{y}^\varepsilon+\sigma\eta^{\frac{1}{\varepsilon}}(\theta_{t}\omega))dt.
\end{array}\right.
\end{eqnarray}
This can be seen by a formal differentiation of $x^\varepsilon$ and $y^\varepsilon-\sigma\eta^{\frac{1}{\varepsilon}}(\omega)$.

For the sake of simplicity, we write $\hat{g}_{i}(\theta_{t}^{\varepsilon}\omega,\hat{x}^\varepsilon,\hat{y}^\varepsilon) = g_{i}(\hat{x}^\varepsilon,\hat{y}^\varepsilon+\sigma\eta^{\frac{1}{\varepsilon}}(\theta_{t}\omega)), i=1, 2.$ Since the additional term $\sigma\eta^{\frac{1}{\varepsilon}}$ doesn't change the Lipschitz constant of the functions on the right hand side, the functions $\hat{g}_{i}$
have the same Lipschitz constant as $g_{i}$.

By hypotheses $(A_{1})-(A_{3})$, system \eqref{3.6} can be solved for any $\omega$ contained in a $\theta$-invariant set $\Omega$ of full measure and for any initial condition $(\hat{x}^\varepsilon(0),\hat{y}^\varepsilon(0))=(x_{0},y_{0})$ such that the cocycle property is
satisfied. Then the solution mapping
\begin{equation}\label{sloution}
(t,\omega,(x_{0},y_{0}))\mapsto\hat{\phi}^\varepsilon(t,\omega,(x_{0},y_{0}))=(\hat{x}^\varepsilon(t,\omega,(x_{0},y_{0})),\hat{y}^\varepsilon(t,\omega,(x_{0},y_{0})))\in\mathbb{R}^{n},
\end{equation}
defines a random dynamical system. In fact, the mapping $\hat{\phi}^\varepsilon$ is
$(\mathcal{B}(\mathbb{R})\otimes\mathcal{F}\otimes\mathcal{B}(\mathbb{R}^{n}),\mathcal{B}(\mathbb{R}^{n}))$-measurable, and for each $\omega\in\Omega$, $\hat{\phi}^\varepsilon(\cdot,\omega):\mathbb{R}\times\mathbb{R}^{n}\rightarrow\mathbb{R}^{n}$
is a Carath\'eodory function.

In the following section we will show that system \eqref{3.6} generates a random dynamical system that has a random slow manifold for sufficiently
small $\varepsilon>0$.
Applying the ideas from the end of Section 2 with $T:=T^{\varepsilon}$ to the
solution of \eqref{3.6}, then system \eqref{3.2} also has a version satisfying the cocycle
property. Clearly,
\begin{align}\nonumber
\varphi^\varepsilon(t,\omega,(x_{0},y_{0}))&=(T^{\varepsilon})^{-1}\big(\theta_{t}\omega,\hat{\varphi}^{\varepsilon}(t,\omega,T^{\varepsilon}(\omega,(x_{0},y_{0})))\big)  \\\label{relationship}
                             &=\hat{\phi}^\varepsilon(t,\omega,(x_{0},y_{0}))+(0,\sigma\eta^{\frac{1}{\varepsilon}}(\theta_{t}\omega)),\ t\in\mathbb{R},\ \omega\in\Omega
\end{align}
is a random dynamical system generated by the original system  \eqref{3.2}. Hence, by the particular structure of $T_{\varepsilon}$ if \eqref{3.6} has a slow
manifold so has \eqref{3.2}.

\section{Random Slow Manifolds }\label{section4}

To study system \eqref{3.6}, for any $\beta\in\mathbb{R}$, we introduce Banach spaces of functions with a geometrically weighted $\sup$ norm \cite{Wa95} as follows:
\begin{eqnarray*}
   C_{\beta}^{s,-}&=&\{\phi^{s,-}:(-\infty,0]\to \mathbb{R}^{n_{1}} \;|\;\phi^{s,-}
\hbox{  is continuous and} \sup_{t\in (-\infty,0]}
|e^{-\beta t }\phi_{t}^{s,-}|< \infty\}, \\
   C_{\beta}^{s,+}&=&\{\phi^{s,+}:[0,\infty)\to \mathbb{R}^{n_{1}} \;|\;\phi^{s,+}
\hbox{  is continuous and }\sup_{t\in [0,\infty)}
|e^{-\beta t }\phi_{t}^{s,+}|< \infty\}
\end{eqnarray*}
with the norms
\[
||\phi^{s,-}||_{ C_{\beta}^{s,-}}:=\sup_{t\in (-\infty,0]}
|e^{-\beta t }\phi_{t}^{s,-}|   ~\mbox{and}~  ||\phi^{s,+}||_{ C_{\beta}^{s,+}} :=\sup_{t\in [0,\infty)}
|e^{-\beta t }\phi_{t}^{s,+}|
\]
Analogously, we define Banach spaces $C_{\beta}^{f,-}$ and $C_{\beta}^{f,+}$ with the norms
\[
||\phi^{f,-}||_{ C_{\beta}^{f,-}}:=\sup_{t\in (-\infty,0]}
|e^{-\beta t }\phi_{t}^{f,-}|  ~\mbox{and}~ ||\phi^{f,+}||_{ C_{\beta}^{f,+}} :=\sup_{t\in [0,\infty)}
|e^{-\beta t }\phi_{t}^{f,+}|.
\]
Let $ C_{\beta}^{\pm}$ be the product space $C_{\beta}^{\pm}:=C_{\beta}^{s,\pm}\times C_{\beta}^{f,\pm}$, $(\phi^{s,\pm},\phi^{f,\pm})\in C_{\beta}^{\pm}$.
$ C_{\beta}^{\pm}$ equipped with the norm
\[
||(\phi^{s,\pm},\phi^{f,\pm})||_{ C_{\beta}^{\pm}}:=||\phi^{s,\pm}||_{ C_{\beta}^{s,\pm}}+||\phi^{f,\pm}||_{ C_{\beta}^{f,\pm}}
\]                                                                                              %
is a Banach space.
\bigskip

Letting $\gamma>0$ satisfy $K<-(\gamma+\gamma_{f})$. For the remainder of the paper, we take $\beta=-\gamma/\varepsilon$ with $\varepsilon>0$ sufficiently small.

\begin{lemma}\label{slow}
Assume that $(A_{1})-(A_{3})$ hold. Then $(x_{0},y_{0})$ is in $\mathcal{M}^{\varepsilon}(\omega)$ if and only if there exists a function
$\hat{\phi}^{\varepsilon}(t)=(\hat{x}^{\varepsilon}(t),\hat{y}^{\varepsilon}(t))=\big(\hat{x}^{\varepsilon}(t,\omega,(x_{0},y_{0})),\hat{y}^{\varepsilon}(t,\omega,(x_{0},y_{0}))\big)\in C_{\beta}^{-}$ with $t\leq0$ such that
\begin{equation}\label{3.9}
\left(
  \begin{array}{ccc}
 \hat{ x}^{\varepsilon}(t)\\
  \hat{y}^{\varepsilon}(t)\\
  \end{array}
\right)
=  \left(
  \begin{array}{ccc}
    & e^{St}x_{0}+\int_{0}^{t}e^{S(t-s)}\hat{g}_{1}(\theta_{s}^{\varepsilon}\omega,\hat{x}^{\varepsilon}(s),\hat{y}^{\varepsilon}(s))ds\\[1ex]
    &\frac{1}{\varepsilon}\int_{-\infty}^{t}e^{\frac{F(t-s)}{\varepsilon}}\hat{g}_{2}(\theta_{s}^{\varepsilon}\omega,\hat{x}^{\varepsilon}(s),\hat{y}^{\varepsilon}(s))ds\\
  \end{array}
\right),
\end{equation}
where
\begin{equation}\label{3.8}
\mathcal{M}^{\varepsilon}(\omega)=\{(x_{0},y_{0})\in\mathbb{R}^{n}:\hat{\phi}^{\varepsilon}(\cdot,\omega,(x_{0},y_{0}))\in C_{\beta}^{-}\},~~ \omega\in\Omega.
\end{equation}
\end{lemma}

\begin{proof}
 If $(x_{0},y_{0})\in \mathcal{M}^{\varepsilon}(\omega)$, by method of constant variation, system \eqref{3.6} is
equivalent to the system of integral equations
 \begin{eqnarray}\label{3.10}
 \left\{\begin{array}{l}
    \hat{x}^{\varepsilon}(t)=e^{St}x_{0}+\int_{0}^{t}e^{S(t-s)}\hat{g}_{1}(\theta_{s}^{\varepsilon}\omega,\hat{x}^{\varepsilon}(s),\hat{y}^{\varepsilon}(s))ds,
    \\[1mm]
    \hat{y}^{\varepsilon}(t)=e^{\frac{F(t-u)}{\varepsilon}}\hat{y}^{\varepsilon}(u)+
    \frac{1}{\varepsilon}\int_{u}^{t}e^{\frac{F(t-s)}{\varepsilon}}\hat{g}_{2}(\theta_{s}^{\varepsilon}\omega,\hat{x}^{\varepsilon}(s),\hat{y}^{\varepsilon}(s))ds
 \end{array}\right.
 \end{eqnarray}
 and $\hat{\phi}^{\varepsilon}(t)\in C_{\beta}^{-}$.
 Moreover, by $u<0$ and $-(\gamma+\gamma_{f})>K>0$, we have
 \begin{equation}\nonumber\begin{array}{l}
|e^{\frac{F(t-u)}{\varepsilon}}\hat{y}^{\varepsilon}(u)|\leq e^{\frac{\gamma_{f}(t-u)}{\varepsilon}}|\hat{y}^{\varepsilon}(u)|
   \leq\sup\limits_{u\in (-\infty,0]}\{e^{-\beta u}|\hat{y}^{\varepsilon}(u)|\}e^{\frac{\gamma_{f}(t-u)}{\varepsilon}}e^{\beta u}\\[2ex]
   =||\hat{y}^{\varepsilon}||_{C_{\beta}^{f,-}}e^{\frac{\gamma_{f}}{\varepsilon}t}e^{(\frac{\gamma_{f}}{\varepsilon}-\beta)(-u)}
   =||\hat{y}^{\varepsilon}||_{C_{\beta}^{f,-}}e^{\frac{\gamma_{f}}{\varepsilon}t}e^{-\frac{\gamma+\gamma_{f}}{\varepsilon}u}
   \to0,\ \text{as}\ u\to-\infty,
\end{array}
\end{equation}
which leads to
\begin{equation}\label{3.12}
\hat{y}^{\varepsilon}(t)=\frac{1}{\varepsilon}\int_{-\infty}^{t}e^{\frac{F(t-s)}{\varepsilon}}\hat{g}_{2}(\theta_{s}^{\varepsilon}\omega,\hat{x}^{\varepsilon}(s),\hat{y}^{\varepsilon}(s))ds.
\end{equation}
Thus \eqref{3.10}-\eqref{3.12} imply that \eqref{3.9} holds.\\
Conversely, let $\hat{\phi}^{\varepsilon}(t,\omega,(x_{0},y_{0}))\in C_{\beta}^{-}$ satisfying \eqref{3.9}, then $(x_{0},y_{0})$ is in $\mathcal{M}^{\varepsilon}(\omega)$ by \eqref{3.8}. Thus, we have finished the proof.
\end{proof}

\begin{lemma}\label{lemma3}
Assume $(A_{1})-(A_{3})$ to be valid. Letting $(\hat{x}^{\varepsilon}(0),\hat{y}^{\varepsilon}(0))=(x_{0},y_{0})$, if there exists an $\delta$ such that $\varepsilon\in(0,\delta)$, the system \eqref{3.9} will have a unique solution $\hat{\phi}^{\varepsilon}(t)=(\hat{x}^{\varepsilon}(t),\hat{y}^{\varepsilon}(t))=\big(\hat{x}^{\varepsilon}(t,\omega,(x_{0},y_{0})),\hat{y}^{\varepsilon}(t,\omega,(x_{0},y_{0}))\big)$ in $C_{\beta}^{-}$.
\end{lemma}
\begin{proof}
For any $\hat{\phi}^{\varepsilon}=(\hat{x}^{\varepsilon},\hat{y}^{\varepsilon})\in C_{\beta}^{-}$, define two operators $\mathcal{I}_{s}^{\varepsilon}:C_{\beta}^{-}\rightarrow C_{\beta}^{s,-}$ and $\mathcal{I}_{f}^{\varepsilon}:C_{\beta}^{-}\rightarrow C_{\beta}^{f,-}$ satisfying
\begin{eqnarray*}
\begin{array}{l}
  \mathcal{I}_{s}^{\varepsilon}(\hat{\phi}^{\varepsilon})[t]=e^{St}x_{0}+\int_{0}^{t}e^{S(t-s)}\hat{g}_{1}(\theta_{s}^{\varepsilon}\omega,\hat{x}^{\varepsilon}(s),
  \hat{y}^{\varepsilon}(s))ds,\\[1mm]
  \mathcal{I}_{f}^{\varepsilon}(\hat{\phi}^{\varepsilon})[t]=\frac{1}{\varepsilon}\int_{-\infty}^{t}e^{\frac{F(t-s)}
  {\varepsilon}}\hat{g}_{2}(\theta_{s}^{\varepsilon}\omega,\hat{x}^{\varepsilon}(s),\hat{y}^{\varepsilon}(s))ds,
  \end{array}
\end{eqnarray*}
and the Lyapunov-Perron transform $\mathcal{I}^{\varepsilon}$ given by
\begin{equation}\label{3.15}
\mathcal{I}^{\varepsilon}(\hat{\phi}^{\varepsilon})[t]
= \left(
  \begin{array}{ccc}
  \mathcal{I}_{s}^{\varepsilon}(\hat{\phi}^{\varepsilon})[t]\\
  \mathcal{I}_{f}^{\varepsilon}(\hat{\phi}^{\varepsilon})[t]\\
  \end{array}
\right).
\end{equation}
Under our assumptions above, $\mathcal{I}^{\varepsilon}$ maps $C_{\beta}^{-}$ into itself. Taking $\hat{\phi}^{\varepsilon}=(\hat{x}^{\varepsilon},\hat{y}^{\varepsilon})\in C_{\beta}^{-}$, then
\begin{align*}
||\mathcal{I}_{s}^{\varepsilon}(\hat{\phi}^{\varepsilon})||_{C_{\beta}^{s,-}}&\leq K\sup_{t\in (-\infty,0]}\Big\{e^{-\beta t}\int_{t}^{0}e^{\gamma_{s}(t-s)}(|\hat{x}^{\varepsilon}(s)|+|\hat{y}^{\varepsilon}(s)+\sigma\eta^{\frac{1}{\varepsilon}}(\theta_{s}\omega)|)ds\Big\}\\
   &\ +\sup_{t\in (-\infty,0]}\{e^{-\beta t}e^{\gamma_{s}t}|x_{0}|\}\\
   &\leq K\sup_{t\in (-\infty,0]}\Big\{\int_{t}^{0}e^{\frac{\gamma+\varepsilon\gamma_{s}}{\varepsilon}(t-s)}ds\Big\}||\hat{\phi}^{\varepsilon}||_{C_{\beta}^{-}}+C_{1}\\
   &\leq\frac{\varepsilon K}{\gamma+\varepsilon\gamma_{s}}||\hat{\phi}^{\varepsilon}||_{C_{\beta}^{-}}+C_{1}
\end{align*}
and
\begin{align*}
||\mathcal{I}_{f}^{\varepsilon}(\hat{\phi}^{\varepsilon})||_{C_{\beta}^{f,-}}&\leq \frac{K}{\varepsilon}\sup_{t\in (-\infty,0]}\Big\{e^{-\beta t}\int_{-\infty}^{t}e^{\frac{\gamma_{f}(t-s)}{\varepsilon}}(|\hat{x}^{\varepsilon}(s)|+|\hat{y}^{\varepsilon}(s)+\sigma\eta^{\frac{1}{\varepsilon}}(\theta_{s}\omega)|)ds\Big\}\\
   &\leq \frac{K}{\varepsilon}\sup_{t\in (-\infty,0]}\Big\{\int_{-\infty}^{t}e^{\frac{\gamma+\gamma_{f}}{\varepsilon}(t-s)}ds\Big\}||\hat{\phi}^{\varepsilon}||_{C_{\beta}^{-}}+C_{2}\\
   &=-\frac{K}{\gamma+\gamma_{f}}||\hat{\phi}^{\varepsilon}||_{C_{\beta}^{-}}+C_{2}.
\end{align*}
Hence, by the definition of $\mathcal{I}^{\varepsilon}$ we obtain
\begin{equation*}
||\mathcal{I}^{\varepsilon}(\hat{\phi}^{\varepsilon})||_{C_{\beta}^{-}}\leq\rho(\varepsilon)||\hat{\phi}^{\varepsilon}||_{C_{\beta}^{-}}+C_{3},
\end{equation*}
where $C_{i}, i=1,2,3$ are constants and
\begin{eqnarray}\label{rho}
\rho(\varepsilon):=\frac{\varepsilon K}{\gamma+\varepsilon\gamma_{s}}- \frac{K}{\gamma+\gamma_{f}}.
\end{eqnarray}

Further, we will show that $\mathcal{I}^{\varepsilon}$ is a contraction.
Let $\hat{\phi}_{1}^{\varepsilon}=(\hat{x}_{1}^{\varepsilon},\hat{y}_{1}^{\varepsilon}), \hat{\phi}_{2}^{\varepsilon}=(\hat{x}_{2}^{\varepsilon},\hat{y}_{2}^{\varepsilon})\in C_{\beta}^{-}$. Using $(A_{1})-(A_{2})$ and the definition of
$C_{\beta}^{-}$, we obtain
{\small\begin{align}\nonumber
||\mathcal{I}_{s}^{\varepsilon}(\hat{\phi}_{1}^{\varepsilon})-\mathcal{I}_{s}^{\varepsilon}(\hat{\phi}_{2}^{\varepsilon})||_{C_{\beta}^{s,-}}&\leq K\sup_{t\in (-\infty,0]}\left\{e^{-\beta t}\int_{t}^{0}e^{\gamma_{s}(t-s)}(|\hat{x}_{1}^{\varepsilon}(s)-\hat{x}_{2}^{\varepsilon}(s)|+|\hat{y}_{1}^{\varepsilon}(s)-\hat{y}_{2}^{\varepsilon}(s)|)ds\right\}\\ \nonumber
   &\leq K\sup_{t\in (-\infty,0]}\left\{\int_{t}^{0}e^{\frac{\gamma+\varepsilon\gamma_{s}}{\varepsilon}(t-s)}ds\right\}||\hat{\phi}_{1}^{\varepsilon}-\hat{\phi}_{2}^{\varepsilon}||_{C_{\beta}^{-}}\\ \label{3.16}
   &\leq\frac{\varepsilon K}{\gamma+\varepsilon\gamma_{s}}||\hat{\phi}_{1}^{\varepsilon}-\hat{\phi}_{2}^{\varepsilon}||_{C_{\beta}^{-}}
\end{align}}
and
{\small\begin{align}\nonumber
||\mathcal{I}_{f}^{\varepsilon}(\hat{\phi}_{1}^{\varepsilon})-\mathcal{I}_{f}^{\varepsilon}(\hat{\phi}_{2}^{\varepsilon})||_{C_{\beta}^{f,-}}&\leq \frac{K}{\varepsilon}\sup_{t\in (-\infty,0]}\left\{e^{-\beta t}\int_{-\infty}^{t}e^{\frac{\gamma_{f}(t-s)}{\varepsilon}}(|\hat{x}_{1}^{\varepsilon}(s)-\hat{x}_{2}^{\varepsilon}(s)|+|\hat{y}_{1}^{\varepsilon}(s)-\hat{y}_{2}^{\varepsilon}(s)|)ds\right \} \\ \nonumber
   &\leq \frac{K}{\varepsilon}\sup_{t\in (-\infty,0]}\left\{\int_{-\infty}^{t}e^{\frac{\gamma+\gamma_{f}}{\varepsilon}(t-s)}ds\right\}||
   \hat{\phi}_{1}^{\varepsilon}-\hat{\phi}_{2}^{\varepsilon}||_{C_{\beta}^{-}}\\ \label{3.17}
   &=-\frac{K}{\gamma+\gamma_{f}}||\hat{\phi}_{1}^{\varepsilon}-\hat{\phi}_{2}^{\varepsilon}||_{C_{\beta}^{-}}.
\end{align}}
By \eqref{3.16} and \eqref{3.17}, we have that
\begin{equation*}
||\mathcal{I}^{\varepsilon}(\hat{\phi}_{1}^{\varepsilon})-\mathcal{I}^{\varepsilon}(\hat{\phi}_{2}^{\varepsilon})||_{C_{\beta}^{-}}
\leq\rho(\varepsilon)||\hat{\phi}_{1}^{\varepsilon}-\hat{\phi}_{2}^{\varepsilon}||_{C_{\beta}^{-}},
\end{equation*}
By \eqref{rho} and hypothesis ($A_{3}$), we have
\begin{equation*}
0<\rho(0)=-\frac{K}{\gamma+\gamma_{f}}<1,\ \ \rho(\varepsilon)=\frac{\gamma K}{(\gamma+\varepsilon\gamma_{s})^{2}}.
\end{equation*}
Then there is a sufficiently small constant $\delta>0$ and a constant $\rho_{0}\in(0,1)$, such that
\begin{equation*}
0<\rho(\varepsilon)\leq\rho_{0}<1\ ~\mbox{for}~\ \varepsilon\in(0,\delta),
\end{equation*}
which implies that $\mathcal{I}^{\varepsilon}$ is strictly contractive.
Let $\hat{\phi}^{\varepsilon}(t)\in C_{\beta}^{-}$ be the unique fixed point, i.e.,
the system \eqref{3.9} has a unique solution $\hat{\phi}^{\varepsilon}(t)$.
\end{proof}
\bigskip
In what follows we investigate the dependence of the fixed point $\hat{\phi}^{\varepsilon}(t)$
of the operator $\mathcal{I}^{\varepsilon}$ on the intial point.
\begin{lemma}\label{lemma4}
Assume the hypotheses of Lemma \ref{lemma3} to be valid. Then for any $(x_{0},y_{0}),(x_{0}^{'},y_{0}^{'})\in\mathbb{R}^{n}$, there is an $\delta>0$ such that if $\varepsilon\in(0,\delta)$, we have
\begin{equation}\label{3.18}
||\hat{\phi}^{\varepsilon}(t,\omega,(x_{0},y_{0}))-\hat{\phi}^{\varepsilon}(t,\omega,(x_{0}^{'},y_{0}^{'}))||_{C_{\beta}^{-}}\leq\frac{|x_{0}-x_{0}^{'}|}{1-\rho(\varepsilon)},
\end{equation}
where $\rho(\varepsilon)$ is defined as \eqref{rho}.
\end{lemma}
\begin{proof}
Taking any $(x_{0},y_{0}),(x_{0}^{'},y_{0}^{'})\in\mathbb{R}^{n}$, for simplicity we write $\hat{\phi}^{\varepsilon}(t,\omega,x_{0})$, $\hat{\phi}^{\varepsilon}(t,\omega,x_{0}^{'})$ instead of $\hat{\phi}^{\varepsilon}(t,\omega,(x_{0},y_{0}))$, $\hat{\phi}^{\varepsilon}(t,\omega,(x_{0}^{'},y_{0}^{'}))$ in the following estimate, respectively. For every $\omega\in\Omega$, we have
$$\begin{array}{rl}
&||\hat{\phi}^{\varepsilon}(t,\omega,x_{0})-\hat{\phi}^{\varepsilon}(t,\omega,x_{0}^{'})||_{C_{\beta}^{-}}\\[1ex]
\leq&\displaystyle\Big|\Big|e^{St}(x_{0}-x_{0}^{'})+\int_{0}^{t}e^{S(t-s)}\Delta \hat{g}_{1}(\theta_{s}^{\varepsilon}\omega,\hat{x}^{\varepsilon}(s),\hat{y}^{\varepsilon}(s))ds\Big|\Big|_{C_{\beta}^{s,-}}\\[1ex]
   &\displaystyle +\Big|\Big|\frac{1}{\varepsilon}\int_{-\infty}^{t}e^{\frac{F(t-s)}{\varepsilon}}\Delta \hat{g}_{2}(\theta_{s}^{\varepsilon}\omega,\hat{x}^{\varepsilon}(s),\hat{y}^{\varepsilon}(s))ds\Big|\Big|_{C_{\beta}^{f,-}}\\[1ex]
\leq&\displaystyle |x_{0}-x_{0}^{'}|+\frac{\varepsilon K}{\gamma+\varepsilon\gamma_{s}}||\hat{\phi}^{\varepsilon}(t,\omega,x_{0})-\hat{\phi}^{\varepsilon}(t,\omega,x_{0}^{'})||_{C_{\beta}^{-}}\\[2ex]
   &\displaystyle -\frac{K}{\gamma+\gamma_{f}}||\hat{\phi}^{\varepsilon}(t,\omega,x_{0})-\hat{\phi}^{\varepsilon}(t,\omega,x_{0}^{'})||_{C_{\beta}^{-}}\\[2ex]
=&|x_{0}-x_{0}^{'}|+\rho(\varepsilon)||\hat{\phi}^{\varepsilon}(t,\omega,x_{0})-\hat{\phi}^{\varepsilon}(t,\omega,x_{0}^{'})||_{C_{\beta}^{-}},
\end{array}$$
where
$$\begin{array}{rl}
  \Delta \hat{g}_{i}=\hat{g}_{i}\big(\theta_{s}^{\varepsilon}\omega,\hat{x}^{\varepsilon}(s,\omega,x_{0}),\hat{y}^{\varepsilon}(s,\omega,x_{0})\big)-\hat{g}_{i}\big(\theta_{s}^{\varepsilon}\omega,\hat{x}^{\varepsilon}(s,\omega,x_{0}^{'}),\hat{y}^{\varepsilon}(s,\omega,x_{0}^{'})\big),\ i=1, 2.
\end{array}$$
Then we obtain
\begin{equation*}
||\hat{\phi}^{\varepsilon}(t,\omega,x_{0})-\hat{\phi}^{\varepsilon}(t,\omega,x_{0}^{'})||_{C_{\beta}^{-}}\leq\frac{1}{1-\rho(\varepsilon)}|x_{0}-x_{0}^{'}|,
\end{equation*}
which completes the proof.
\end{proof}

\medskip
By Lemma \ref{slow}, Lemma \ref{lemma3} and Lemma \ref{lemma4}, we can construct the slow manifold as a random graph.

\begin{theorem}\label{slowmanifold}
Assume that $(A_{1})-(A_{3})$ hold and that $\varepsilon$ is sufficiently small.
Then the system \eqref{3.6} has a invariant manifold
$\mathcal{M}^{\varepsilon}(\omega)=\{(x_{0},\hat{h}^{\varepsilon}(\omega,x_{0})): x_{0}\in\mathbb{R}^{n_{1}}\}$, where
$\hat{h}^{\varepsilon}(\cdot,\cdot):\Omega\times\mathbb{R}^{n_{1}}\mapsto\mathbb{R}^{n_{2}}$ is a Lipschitz continuous function
with Lipschitz constant
satisfying
\begin{equation}\label{3.19}
{\rm Lip}\hat{h}^{\varepsilon}(\omega,\cdot)\leq-\frac{K}{\gamma+\gamma_{f}}\frac{1}{1-\rho(\varepsilon)},\ \omega\in\Omega.
\end{equation}
\end{theorem}

\begin{proof}
Taking any $x_{0}\in\mathbb{R}^{n_{1}}$, define the Lyapunov-Perron transform
\begin{equation}\label{3.20}
\hat{h}^{\varepsilon}(\omega,x_{0})= \frac{1}{\varepsilon}\int_{-\infty}^{0}e^{\frac{-Fs}{\varepsilon}}\hat{g}_{2}(\theta_{s}^{\varepsilon}\omega,\hat{x}^{\varepsilon}(s),\hat{y}^{\varepsilon}(s))ds
\end{equation}
 where $(\hat{x}^{\varepsilon}(s),\hat{ y}^{\varepsilon}(s))$ is the unique solution in $C_{\beta}^{-}$ of the system \eqref{3.9} with $s\leq0$.
It follows from Lemma \ref{slow}, Lemma \ref{lemma3}, \eqref{3.8} and \eqref{3.20} that
\begin{equation}\label{M}
\mathcal{M}^{\varepsilon}(\omega)=\left\{(\hat{x}_{0},\hat{h}^{\varepsilon}(\omega,x_{0})):x_{0}\in\mathbb{R}^{n_{1}}\right\}.
\end{equation}
By \eqref{3.17} and Lemma \ref{lemma4}, we have
\begin{equation*}
|\hat{h}^{\varepsilon}(\omega,x_{0})-\hat{h}^{\varepsilon}(\omega,x_{0}^{'})|\leq-\frac{K}{\gamma+\gamma_{f}}\frac{1}{1-\rho(\varepsilon)}|x_{0}-x_{0}^{'}|
\end{equation*}
for all $x_{0}, x_{0}^{'}\in\mathbb{R}^{n_{1}}$, $\omega\in\Omega$.

From Section \ref{s2}, $\mathcal{M}^{\varepsilon}(\omega)$
is a random set. Now we are going to prove that $\mathcal{M}^{\varepsilon}(\omega)$ is invariant in the following sense
\begin{equation*}
  \hat{\phi}^{\varepsilon}(s,\omega,\mathcal{M}^{\varepsilon}(\omega))\subset\mathcal{M}^{\varepsilon}(\theta_{s}^{\varepsilon}\omega)\ \ \text{for}\ \ s\geq0.
\end{equation*}
In other words, for each $(x_{0},y_{0})\in\mathcal{M}^{\varepsilon}(\omega)$, we have $\hat{\phi}^{\varepsilon}(s,\omega,(x _{0},y_{0}))\in\mathcal{M}^{\varepsilon}(\theta_{s}^{\varepsilon}\omega)$.
Using the cocycle property
\begin{equation*}
  \hat{\phi}^{\varepsilon}(\cdot+s,\omega,(x_{0},y_{0}))=\hat{\phi}^{\varepsilon}(\cdot,\theta_{s}\omega,\hat{\phi}^{\varepsilon}(s,\omega,(x _{0},y_{0})))
\end{equation*}
and the fact $\hat{\phi}^{\varepsilon}(\cdot,\omega,(x_{0},y_{0}))\in C_{\beta}^{-}$, it follows  that $\hat{\phi}^{\varepsilon}(\cdot,\theta_{s}\omega,\hat{\phi}^{\varepsilon}(s,\omega,(x_{0},y_{0})))\in C_{\beta}^{-}$. Thus, $\hat{\phi}^{\varepsilon}(s,\omega,(x_{0},y_{0}))\in\mathcal{M}^{\varepsilon}(\theta_{s}^{\varepsilon}\omega)$. This completes the proof.
\end{proof}
\begin{remark}
The invariant manifold $\mathcal{M}^{\varepsilon}(\omega)$ is independent of the choice of $\gamma$.
\end{remark}

Furthermore, the invariant manifold $\mathcal{M}^{\varepsilon}(\omega)$ exponentially attract other dynamical orbits. Hence, $\mathcal{M}^{\varepsilon}(\omega)$ is a slow manifold.
\begin{theorem}\label{Etp}
Assume that $(A_{1})-(A_{3})$ hold.
Then the invariant manifold $\mathcal{M}^{\varepsilon}(\omega)=\{(x_{0},\hat{h}^{\varepsilon}(\omega,x_{0})):x_{0}\in\mathbb{R}^{n_{1}}\}$
for slow-fast random system \eqref{3.6}  obtained in  Theorem \ref{slowmanifold} has exponential tracking property in the following sense: For any $z_{0}=(x_{0},y_{0})\in\mathbb{R}^{n}$, there is a $z_{0}^{'}=(x_{0}^{'},y_{0}^{'})\in\mathcal{M}^{\varepsilon}(\omega)$ such that
\begin{equation*}
  |\hat{\phi}^{\varepsilon}(t,\omega,z_{0}^{'})-\hat{\phi}^{\varepsilon}(t,\omega,z_{0})|\leq Ce^{-ct}|z_{0}^{'}-z_{0}|,\ \ t\geq0,
\end{equation*}
where $C>0$ and $c>0$.
\end{theorem}
\begin{proof}
Let
$$\begin{array}{rl}
&\hat{\phi}^{\varepsilon}(t,\omega,(x_{0},y_{0}))=\big(\hat{x}^{\varepsilon}(t,\omega,(x_{0},y_{0})),\hat{y}^{\varepsilon}(t,\omega,(x_{0},y_{0}))\big)\\[1ex]
&\hat{\phi}^{\varepsilon}(t,\omega,(x_{0}^{'},y_{0}^{'}))=\big(\hat{x}^{\varepsilon}(t,\omega,(x_{0}^{'},y_{0}^{'})),\hat{y}^{\varepsilon}(t,\omega,(x_{0}^{'},y_{0}^{'}))\big)
\end{array}$$
be the two dynamical orbits of system \eqref{3.6} with the initial condition
\begin{eqnarray*}
  \hat{\phi}^{\varepsilon}(0,\omega,(x_{0},_{0}))=(x_{0},y_{0}),~~
  \hat{\phi}^{\varepsilon}(0,\omega,(x_{0}^{'},y_{0}^{'}))=(x_{0}^{'},y_{0}^{'}).
\end{eqnarray*}
Then
{\small\begin{align*}
\psi^{\varepsilon}(t)&=\hat{\phi}^{\varepsilon}(t,\omega,(x_{0}^{'},y_{0}^{'}))-\hat{\phi}^{\varepsilon}(t,\omega,(x_{0},y_{0}))\\
                     &=\big(\hat{x}^{\varepsilon}(t,\omega,(x_{0}^{'},y_{0}^{'}))-\hat{x}^{\varepsilon}(t,\omega,(x_{0},y_{0})),
                      \hat{y}^{\varepsilon}(t,\omega,(x_{0}^{'},y_{0}^{'}))-\hat{y}^{\varepsilon}(t,\omega,(x_{0},y_{0}))\big)\\
                     &:=\big(u^{\varepsilon}(t),v^{\varepsilon}(t)\big)
     \end{align*}}
satisfies the equation
\begin{eqnarray}\label{3.23}
\left\{\begin{array}{l}
   du^{\varepsilon}=Su^{\varepsilon}dt+\Delta \hat{g}_{1}(\theta_{t}^{\varepsilon}\omega,u^{\varepsilon},v^{\varepsilon})dt,\\
   dv^{\varepsilon}=\frac{1}{\varepsilon}Fv^{\varepsilon}dt+\frac{1}{\varepsilon}\Delta \hat{g}_{2}(\theta_{t}^{\varepsilon}\omega,u^{\varepsilon},v^{\varepsilon})dt
   \end{array}\right.
\end{eqnarray}
with the nonlinear items
{\small\begin{align}\nonumber
\Delta \hat{g}_{i}(\theta_{t}^{\varepsilon}\omega,u^{\varepsilon},v^{\varepsilon})&=\hat{g}_{i}\big(\theta_{t}^{\varepsilon}\omega,y^{\varepsilon}(t)+\hat{x}^{\varepsilon}(t,\omega,(x_{0},y_{0})),v^{\varepsilon}(t)+\hat{y}^{\varepsilon}(t,\omega,(x_{0},y_{0}))\big)\\\label{3.25}
    &\ \ \ -\hat{g}_{i}\big(\theta_{t}^{\varepsilon}\omega,\hat{x}^{\varepsilon}(t,\omega,(x_{0},y_{0})),\hat{y}^{\varepsilon}(t,\omega,(x_{0},y_{0}))\big),\ i=1,\ 2,
\end{align}}
and the initial condition
\begin{eqnarray*}
(u^{\varepsilon}(0),v^{\varepsilon}(0)) = (u_{0},v_{0}) = (x_{0}^{'}-x_{0},y_{0}^{'}-y_{0}).
\end{eqnarray*}
By direct calculation, for $t\geq0$, $\psi^{\varepsilon}(t)=(u^{\varepsilon}(t),v^{\varepsilon}(t))$ satisfying
\begin{equation}\label{3.27}
\left(
  \begin{array}{ccc}
  u^{\varepsilon}(t)\\
  v^{\varepsilon}(t)\\
  \end{array}
\right)
=  \left(\!\!\!\!\!\!
  \begin{array}{lll}
    &\int_{+\infty}^{t}e^{S(t-s)}\Delta \hat{g}_{1}(\theta_{s}^{\varepsilon}\omega,u^{\varepsilon}(s),v^{\varepsilon}(s))ds\\[1ex]
    & e^{\frac{Ft}{\varepsilon}}v_{0}+\frac{1}{\varepsilon}\int_{0}^{t}e^{\frac{F(t-s)}{\varepsilon}}\Delta \hat{g}_{2}(\theta_{s}^{\varepsilon}\omega,u^{\varepsilon}(s),v^{\varepsilon}(s))ds\\
  \end{array}\!\!
\right)
\end{equation}
 is a solution of \eqref{3.23} in $C_{\beta}^{+}$.

Now we use the Lyapunov-Perron transform again, to prove that \eqref{3.27} has a unique solution $(u^{\varepsilon}(t),v^{\varepsilon}(t))$ in $C_{\beta}^{+}$ with $(x_{0}^{'},y_{0}^{'})=(u_{0},v_{0})+(x_{0},y_{0})\in\mathcal{M}^{\varepsilon}(\omega)$. Clearly,
\begin{eqnarray}\nonumber
&&(x_{0}^{'},y_{0}^{'})\in\mathcal{M}^{\varepsilon}(\omega)\\\nonumber
&\Longleftrightarrow& y_{0}^{'}= \frac{1}{\varepsilon}\int_{-\infty}^{0}e^{\frac{-Fs}{\varepsilon}}\hat{g}_{2}(\theta_{s}^{\varepsilon}\omega,
\hat{x}^{\varepsilon}(s,\omega,x_{0}^{'}),\hat{y}^{\varepsilon}(s,\omega,x_{0}^{'}))ds\\\nonumber
&\Longleftrightarrow& v_{0}=-y_{0}+\frac{1}{\varepsilon}\int_{-\infty}^{0}e^{\frac{-Fs}{\varepsilon}}\hat{g}_{2}(\theta_{s}^{\varepsilon}\omega,
\hat{x}^{\varepsilon}(s,\omega,u_{0}+x_{0}),\hat{y}^{\varepsilon}(s,\omega,u_{0}+x_{0}))ds  \\\label{3.28} &&\hspace{5mm}=-y_{0}+\hat{h}^{\varepsilon}(\omega,u_{0}+x_{0}).
\end{eqnarray}
Taking $\psi^{\varepsilon}=(u^{\varepsilon},v^{\varepsilon})\in C_{\beta}^{+}$, define two operators $\mathcal{J}_{s}^{\varepsilon}:C_{\beta}^{+}\rightarrow C_{\beta}^{s,+}$ and $\mathcal{J}_{f}^{\varepsilon}:C_{\beta}^{+}\rightarrow C_{\beta}^{f,+}$ satisfying
\begin{eqnarray*}
\begin{array}{l}
 \mathcal{J}_{s}^{\varepsilon}(\psi^{\varepsilon})[t]=\int_{+\infty}^{t}e^{S(t-s)}\Delta \hat{g}_{1}(\theta_{s}^{\varepsilon}\omega,u^{\varepsilon}(s),v^{\varepsilon}(s))ds,\\[1mm]
 \mathcal{J}_{f}^{\varepsilon}(\psi^{\varepsilon})[t]= e^{\frac{Ft}{\varepsilon}}v_{0}+\frac{1}{\varepsilon}\int_{0}^{t}e^{\frac{F(t-s)}{\varepsilon}}\Delta \hat{g}_{2}(\theta_{s}^{\varepsilon}\omega,u^{\varepsilon}(s),v^{\varepsilon}(s))ds.
  \end{array}
\end{eqnarray*}
Moreover, the Lyapunov-Perron transform $\mathcal{J}^{\varepsilon}:C_{\beta}^{+}\rightarrow C_{\beta}^{+}$ is given by
\begin{equation}\label{3.31}
\mathcal{J}^{\varepsilon}(\psi^{\varepsilon})[t]
= \left(
  \begin{array}{ccc}
  \mathcal{J}_{s}^{\varepsilon}(\psi^{\varepsilon})[t]\\
  \mathcal{J}_{f}^{\varepsilon}(\psi^{\varepsilon})[t]\\
  \end{array}
\right).
\end{equation}
We have the following estimates
{\small\begin{align*}
&||\mathcal{J}_{s}^{\varepsilon}(\psi^{\varepsilon})||_{C_{\beta}^{s,+}}\\
&=\Big|\Big|\int_{+\infty}^{t}e^{S(t-s)} \big[\hat{g}_{1}(\theta_{s}^{\varepsilon}\omega,u^{\varepsilon}(s)+\hat{x}^{\varepsilon}(s),v^{\varepsilon}(s)+\hat{y}^{\varepsilon}(s))-\hat{g}_{1}(\theta_{s}^{\varepsilon}\omega,\hat{x}^{\varepsilon}(s),\hat{y}^{\varepsilon}(s))\big]ds\Big|\Big|_{C_{\beta}^{s,+}}\\
   &\leq K\sup_{t\in [0,\infty)}\Big\{e^{-\beta t}\int_{t}^{+\infty}e^{\gamma_{s}(t-s)}(|u^{\varepsilon}(s)|+|v^{\varepsilon}(s)|)ds\Big\}\\
   &\leq K\sup_{t\in [0,\infty)}\Big\{\int_{t}^{+\infty}e^{\frac{\gamma+\varepsilon\gamma_{s}}{\varepsilon}(t-s)}ds\Big\}||\psi^{\varepsilon}||_{C_{\beta}^{+}}\\
   &\leq\frac{\varepsilon K}{\gamma+\varepsilon\gamma_{s}}||\psi^{\varepsilon}||_{C_{\beta}^{+}},
\end{align*}}
and
{\small\begin{align*}
||\mathcal{J}_{f}^{\varepsilon}(\psi^{\varepsilon})||_{C_{\beta}^{f,+}}
&=\Big|\Big|e^{\frac{Ft}{\varepsilon}}v_{0}+\frac{1}{\varepsilon}\int_{0}^{t}e^{\frac{F(t-s)}{\varepsilon}} \big[\hat{g}_{2}(\theta_{s}^{\varepsilon}\omega,u^{\varepsilon}(s)+\hat{x}^{\varepsilon}(s),v^{\varepsilon}(s)+\hat{y}^{\varepsilon}(s))\\
&\ \ \ \ -\hat{g}_{2}(\theta_{s}^{\varepsilon}\omega,\hat{x}^{\varepsilon}(s),\hat{y}^{\varepsilon}(s))\big]ds\Big|\Big|_{C_{\beta}^{f,+}}\\
&\leq \frac{K}{\varepsilon}\sup_{t\in [0,\infty)}\Big\{e^{-\beta t}\int_{0}^{t}e^{\frac{\gamma_{f}(t-s)}{\varepsilon}}(|u^{\varepsilon}(s)|+|v^{\varepsilon}(s)|)ds\Big\}+\sup_{t\in [0,\infty)}\big\{e^{-\beta t}e^{\frac{\gamma_{f}}{\varepsilon}t}|v_{0}|\big\}\\
&\leq \frac{K}{\varepsilon}\sup_{t\in [0,\infty)}\Big\{\int_{0}^{t}e^{\frac{\gamma+\gamma_{f}}{\varepsilon}(t-s)}ds\Big\}||\psi^{\varepsilon}||_{C_{\beta}^{+}}+|v_{0}|\\
&=-\frac{K}{\gamma+\gamma_{f}}||\psi^{\varepsilon}||_{C_{\beta}^{+}}+|v_{0}|.
\end{align*}}
Hence, by \eqref{3.31}, we obtain
\begin{equation}\label{3.32}
||\mathcal{J}^{\varepsilon}(\psi^{\varepsilon})||_{C_{\beta}^{+}}\leq\rho(\varepsilon)||\psi^{\varepsilon}||_{C_{\beta}^{+}}+|v_{0}|
\end{equation}
where $\rho(\varepsilon)$ is defined as \eqref{rho} in the proof of Lemma \ref{lemma3}.

For any $\psi^{\varepsilon}=(u^{\varepsilon},v^{\varepsilon}),\ \bar{\psi}^{\varepsilon}=(\bar{u}^{\varepsilon},\bar{v}^{\varepsilon})\in C_{\beta}^{+}$,
{\small\begin{equation}\label{3.33}
\begin{array}{rl}
&||\mathcal{J}_{s}^{\varepsilon}(\psi^{\varepsilon})-\mathcal{J}_{s}^{\varepsilon}(\bar{\psi}^{\varepsilon})||_{C_{\beta}^{s,+}}\\[1ex]
=&\displaystyle\Big|\Big|\int_{+\infty}^{t}e^{S(t-s)}\big[\Delta  \hat{g}_{1}(\theta_{s}^{\varepsilon}\omega,u^{\varepsilon}(s),v^{\varepsilon}(s))
-\Delta \hat{g}_{1}(\theta_{s}^{\varepsilon}\omega,\bar{u}^{\varepsilon}(s),\bar{v}^{\varepsilon}(s))\big]ds\Big|\Big|_{C_{\beta}^{s,+}}\\[1ex]
=&\displaystyle\Big|\Big|\int_{+\infty}^{t}e^{S(t-s)}\big[ \hat{g}_{1}(\theta_{s}^{\varepsilon}\omega,u^{\varepsilon}(s)+\hat{x}^{\varepsilon}(s),v^{\varepsilon}(s)+\hat{y}^{\varepsilon}(s))\\
&- \hat{g}_{1}(\theta_{s}^{\varepsilon}\omega,\bar{u}^{\varepsilon}(s)+\hat{x}^{\varepsilon}(s),\bar{v}^{\varepsilon}(s)+\hat{y}^{\varepsilon}(s))\big]ds\Big|\Big|_{C_{\beta}^{s,+}}\\[1ex]
\leq&\displaystyle K\sup_{t\in[0,\infty)}\Big\{e^{-\beta t}\int_{t}^{+\infty}e^{\gamma_{s}(t-s)}(|u^{\varepsilon}(s)-\bar{u}^{\varepsilon}(s)|+|v^{\varepsilon}(s)-\bar{v}^{\varepsilon}(s)|)ds\Big\}\\[1ex]
\leq&\displaystyle K\sup_{t\in[0,\infty)}\Big\{\int_{t}^{+\infty}e^{\frac{\gamma+\varepsilon\gamma_{s}}{\varepsilon}(t-s)}ds\Big\}||\psi^{\varepsilon}-\bar{\psi}^{\varepsilon}||_{C_{\beta}^{+}}\\[2ex]
=&\displaystyle \frac{\varepsilon K}{\gamma+\varepsilon\gamma_{s}}||\psi^{\varepsilon}-\bar{\psi}^{\varepsilon}||_{C_{\beta}^{+}}.
\end{array}
\end{equation}}
On the one hand, by \eqref{3.28}, we have
{\small\begin{equation*}
\begin{array}{rl}
&|e^{\frac{Ft}{\varepsilon}}(v_{0}-\bar{v}_{0})|\leq e^{\frac{\gamma_{f}}{\varepsilon}t}{\rm Lip}\hat{h}^{\varepsilon}|u_{0}-\bar{u}_{0}|\\[1ex]
\leq&\displaystyle e^{\frac{\gamma_{f}}{\varepsilon}t}{\rm Lip}\hat{h}^{\varepsilon}\Big|\int_{+\infty}^{0}e^{S(-s)}\big[\Delta \hat{g}_{1}(\theta_{s}^{\varepsilon}\omega,u^{\varepsilon}(s),v^{\varepsilon}(s))-\Delta \hat{g}_{1}(\theta_{s}^{\varepsilon}\omega,\bar{u}^{\varepsilon}(s),\bar{v}^{\varepsilon}(s))\big]ds\Big|\\[2ex]
=&\displaystyle e^{\frac{\gamma_{f}}{\varepsilon}t}{\rm Lip}\hat{h}^{\varepsilon}\Big|\int_{+\infty}^{0}e^{S(-s)}\big[ \hat{g}_{1}(\theta_{s}^{\varepsilon}\omega,u^{\varepsilon}(s)+\hat{x}^{\varepsilon}(s),v^{\varepsilon}(s)+\hat{y}^{\varepsilon}(s))\\[1ex]
\displaystyle&\ \ \ - \hat{g}_{1}(\theta_{s}^{\varepsilon}\omega,\bar{u}^{\varepsilon}(s)+\hat{x}^{\varepsilon}(s),\bar{v}^{\varepsilon}(s)+\hat{y}^{\varepsilon}(s))\big]ds\Big|\\[1ex]
\leq&\displaystyle{\rm Lip}\hat{h}^{\varepsilon}\cdot Ke^{\frac{\gamma_{f}}{\varepsilon}t}\int_{0}^{+\infty}e^{\gamma_{s}(-s)}|\psi^{\varepsilon}(s)-\bar{\psi}^{\varepsilon}(s)|ds,
\end{array}
\end{equation*}}
which leads to
{\small\begin{align}\nonumber
||e^{\frac{Ft}{\varepsilon}}(v_{0}-\bar{v}_{0})||_{C_{\beta}^{f,+}}&\leq{\rm Lip}\hat{h}^{\varepsilon}\cdot K||\psi^{\varepsilon}(s)-\bar{\psi}^{\varepsilon}(s)||_{C_{\beta}^{+}}\sup_{t\in[0,\infty)}
\Big\{e^{(\frac{\gamma_{f}}{\varepsilon}-\beta)t}\int_{0}^{+\infty}e^{(\beta-\gamma_{s})s}ds\Big\}\\\label{varepsilon1}
   &\leq\frac{\varepsilon K{\rm Lip}\hat{h}^{\varepsilon}}{\gamma+\varepsilon\gamma_{s}}||\psi^{\varepsilon}(s)-\bar{\psi}^{\varepsilon}(s)||_{C_{\beta}^{+}}.
\end{align}}
On the other hand, we observe
{\small\begin{equation}
\begin{array}{rl}
&\displaystyle\Big|\Big|\frac{1}{\varepsilon}\int_{0}^{t}e^{\frac{F(t-s)}{\varepsilon}}\big[\Delta \hat{g}_{2}(\theta_{s}^{\varepsilon}\omega,u^{\varepsilon}(s),v^{\varepsilon}(s))-\Delta \hat{g}_{2}(\theta_{s}^{\varepsilon}\omega,\bar{u}^{\varepsilon}(s),\bar{v}^{\varepsilon}(s))\big]ds\Big|\Big|_{C_{\beta}^{f,+}}\\
=&\displaystyle\Big|\Big|\frac{1}{\varepsilon}\int_{0}^{t}e^{\frac{F(t-s)}{\varepsilon}}\big[\hat{g}_{2}(\theta_{s}^{\varepsilon}\omega,u^{\varepsilon}(s)+\hat{x}^{\varepsilon}(s),v^{\varepsilon}(s)+\hat{y}^{\varepsilon}(s))\\[1ex]
&-\hat{g}_{2}(\theta_{s}^{\varepsilon}\omega,\bar{u}^{\varepsilon}(s)+\hat{x}^{\varepsilon}(s),\bar{v}^{\varepsilon}(s)+\hat{y}^{\varepsilon}(s))\big]ds\Big|\Big|_{C_{\beta}^{f,+}}\\[1ex]
\leq&\displaystyle \frac{K}{\varepsilon}\sup_{t\in[0,\infty)}\Big\{e^{-\beta t}\int_{0}^{t}e^{\frac{\gamma_{f}(t-s)}{\varepsilon}}(|u^{\varepsilon}(s)-\bar{u}^{\varepsilon}(s)|+|v^{\varepsilon}(s)-\bar{v}^{\varepsilon}(s)|)ds\Big\}\\[2ex]
\leq&\displaystyle \frac{K}{\varepsilon}\sup_{t\in[0,\infty)}\Big\{\int_{0}^{t}e^{\frac{\gamma+\gamma_{f}}{\varepsilon}(t-s)}ds\Big\}||\psi^{\varepsilon}-\bar{\psi}^{\varepsilon}||_{C_{\beta}^{+}}\\[2ex]
=&\displaystyle-\frac{K}{\gamma+\gamma_{f}}||\psi^{\varepsilon}-\bar{\psi}^{\varepsilon}||_{C_{\beta}^{+}}.
\end{array}
\label{varepsilon2}
\end{equation}}
Using  \eqref{3.19}, \eqref{varepsilon1} and \eqref{varepsilon2}, it follows that
{\small\begin{equation}\label{3.34}
\begin{array}{rl}
&||\mathcal{J}_{f}^{\varepsilon}(\psi^{\varepsilon})-\mathcal{J}_{f}^{\varepsilon}(\bar{\psi}^{\varepsilon})||_{C_{\beta}^{f,+}}\\
\leq&\displaystyle||e^{\frac{Ft}{\varepsilon}}(v_{0}-\bar{v}_{0})||_{C_{\beta}^{f,+}}+\Big|\Big|\frac{1}{\varepsilon}\int_{0}^{t}e^{\frac{F(t-s)}{\varepsilon}}\big[\Delta \hat{g}_{2}(u^{\varepsilon}(s),v^{\varepsilon}(s),\theta_{s}^{\varepsilon}\omega)\\
&\ \ -\Delta \hat{g}_{2}(\bar{u}^{\varepsilon}(s),\bar{v}^{\varepsilon}(s),\theta_{s}^{\varepsilon}\omega)\big]ds\Big|\Big|_{C_{\beta}^{f,+}}\\
\leq&\displaystyle(\frac{\varepsilon K{\rm Lip}\hat{h}^{\varepsilon}}{\gamma+\varepsilon\gamma_{s}}-\frac{K}{\gamma+\gamma_{f}})||\psi^{\varepsilon}-\bar{\psi}^{\varepsilon}||_{C_{\beta}^{+}}\\
\leq&\displaystyle-\Big(\frac{\varepsilon K^{2}}{(\gamma+\varepsilon\gamma_{s})(\gamma+\gamma_{f})[1-K(\frac{\varepsilon }{\gamma+\varepsilon\gamma_{s}}-\frac{1}{\gamma+\gamma_{f}})]}+\frac{K}{\gamma+\gamma_{f}}\Big)||\psi^{\varepsilon}-\bar{\psi}^{\varepsilon}||_{C_{\beta}^{+}}.
\end{array}
\end{equation}}
By \eqref{3.33} and \eqref{3.34}, we have
\begin{equation*}
||\mathcal{J}^{\varepsilon}(\psi^{\varepsilon})-\mathcal{J}^{\varepsilon}(\bar{\psi}^{\varepsilon})||_{C_{\beta}^{+}}\leq\bar{\rho}(\varepsilon)||\psi^{\varepsilon}-\bar{\psi}^{\varepsilon}||_{C_{\beta}^{+}}
\end{equation*}
with
\begin{align}\nonumber
   \bar{\rho}(\varepsilon)&=\frac{\varepsilon K}{\gamma+\varepsilon\gamma_{s}}-\frac{K}{\gamma+\gamma_{f}}
   -\frac{\varepsilon K^{2}}{(\gamma+\varepsilon\gamma_{s})(\gamma+\gamma_{f})[1-K(\frac{\varepsilon }{\gamma+\varepsilon\gamma_{s}}-\frac{1}{\gamma+\gamma_{f}})]}.
\end{align}
Note that $\displaystyle0<K<-(\gamma+\gamma_{f})$. Clearly, for small $\epsilon$, we have $0<\bar{\rho}(\varepsilon)<1$,
which implies that $\mathcal{J}^{\varepsilon}$ is a contraction in $C_{\beta}^{+}$.
Thus, there is a unique fixed point $\psi^{\varepsilon}:=(u^{\varepsilon},v^{\varepsilon})$ in $C_{\beta}^{+}$.
Further, $\psi^{\varepsilon}$  satisfies $(x_{0}^{'},y_{0}^{'})=(u_{0},v_{0})+(x_{0},y_{0})\in\mathcal{M}^{\varepsilon}(\omega)$. In fact, the solution of \eqref{3.27} in $C_{\beta}^{+}$ if and only if it is a fixed point of the Lyapunov-Perron transform \eqref{3.31}.
Moreover, it follows from \eqref{3.32} that
\begin{equation*}
||\psi^{\varepsilon}(\cdot)||_{C_{\beta}^{+}}\leq\frac{1}{1-\rho(\varepsilon)}|v_{0}|
\end{equation*}
which leads to
\begin{equation}\label{3.35}
||\hat{\phi}^{\varepsilon}(t,\omega,(x_{0}^{'},y_{0}^{'}))-\hat{\phi}^{\varepsilon}(t,\omega,(x_{0},y_{0}))||_{C_{\beta}^{+}}\leq\frac{1}{1-\rho(\varepsilon)}|y_{0}^{'}-y_{0}|.
\end{equation}
And then
\begin{equation*}
|\hat{\phi}^{\varepsilon}(t,\omega,z_{0}^{'})-\hat{\phi}^{\varepsilon}(t,\omega,z_{0})|\leq\frac{e^{-\frac{\gamma}{\varepsilon} t}}{1-\rho(\varepsilon)}|z_{0}^{'}-z_{0}|,
\end{equation*}
with $t\geq0$ and $-\gamma/\varepsilon<0$. Hence,   the proof has been finished.
\end{proof}
\begin{remark}
  $0<\rho(\varepsilon)<1$ and $0<\bar{\rho}(\varepsilon)<1$
are the critical points in the proof of Lemma \ref{lemma3} and Theorem \ref{Etp} respectively.
\end{remark}
According to Theorem \ref{slowmanifold} and Theorem \ref{Etp}, the system \eqref{3.6}
has an exponential tracking slow manifold. By the relationship between of $\varphi^\varepsilon(t,\omega)$ and $\hat{\phi}^\varepsilon(t,\omega)$, so has the slow-fast system \eqref{3.2}.
\begin{theorem}\label{yuan}
Assume that $(A_{1})-(A_{3})$ hold. The slow-fast system \eqref{3.2} with jumps has a slow manifold
\begin{align*}
  M^{\varepsilon}(\omega)=(T^{\varepsilon})^{-1}\mathcal{M}^{\varepsilon}(\omega)=\mathcal{M}^{\varepsilon}(\omega)+(0,\sigma\eta^{\frac{1}{\varepsilon}}(\omega))
 =\{(x_{0},h^{\varepsilon}(\omega,x_{0})):x_{0}\in\mathbb{R}^{n_{1}}\}
\end{align*}
with $h^{\varepsilon}(\omega,x_{0})=\hat{h}^{\varepsilon}(\omega,x_{0})+\sigma\eta^{\frac{1}{\varepsilon}}(\theta_{t}\omega)$, where $T^{\varepsilon}$ is defined in \eqref{T}.
\end{theorem}
\begin{proof}
We have the relationship between $\varphi$ and $\hat{\varphi}$ given by \eqref{relationship}:
\begin{align*}
   \varphi^{\varepsilon}(t,\omega,M^{\varepsilon}(\omega))&=(T^{\varepsilon})^{-1}\big(\theta_{t}\omega,\hat{\varphi}^{\varepsilon}(t,\omega,T^{\varepsilon}(\omega,M^{\varepsilon}(\omega)))\big)  \\
   &=(T^{\varepsilon})^{-1}\big(\theta_{t}\omega,\hat{\varphi}^{\varepsilon}(t,\omega,\mathcal{M}^{\varepsilon}(\omega))\big)\subset(T^{\varepsilon})^{-1}\big(\theta_{t}\omega,\mathcal{M}^{\varepsilon}(\theta_{t}\omega)\big)=M^{\varepsilon}(\theta_{t}\omega).
\end{align*}
Note that $t\rightarrow\eta^{\frac{1}{\varepsilon}}(\theta_{t}\omega)$ has a sublinear growth rate for $1<\alpha<2$; see \cite{DLS03,L10}. Thus the
transform $(T^{\varepsilon})^{-1}(\theta_{t}\omega)$ does not change the exponential tracking property. It follows that  $M^{\varepsilon}(\omega)$ is a slow manifold.
\end{proof}
\begin{remark}
It is worthy mentioning that the dynamical orbits in $ M^{\varepsilon}(\omega)$ are c\'adl\'ag and adapted.
\end{remark}
\begin{corollary}\label{reduction}
Assume that $(A_{1})-(A_{3})$ hold and that $\varepsilon$ is sufficiently small. For any solution $\phi^{\varepsilon}(t,\omega)=(x^{\varepsilon}(t,\omega),y^{\varepsilon}(t,\omega))$ with initial condition
$z_{0}=(x_{0},y_{0})$ to the fast-slow system \eqref{3.2}, there exists a solution $\bar{\phi}^{\varepsilon}(t,\omega)=(\bar{x}^{\varepsilon}(t,\omega),\bar{y}^{\varepsilon}(t,\omega))$ with initial point $\bar{z}_{0}=(\bar{x}_{0},\bar{y}_{0})$ on the manifold $M^{\varepsilon}(\omega)$ which satisfies the reduction system
\begin{equation}\label{3.36}
\displaystyle d\bar{x}^{\varepsilon}=S\bar{x}^{\varepsilon}dt+g_{1}(\bar{x}^{\varepsilon},h^{\varepsilon}(\theta_{t}\omega,\bar{x}^{\varepsilon}))dt
\end{equation}
such that
\begin{equation*}
|\varphi^{\varepsilon}(t,\omega)-\bar{\varphi}^{\varepsilon}(t,\omega)|\leq\frac{e^{-\frac{\gamma}{\varepsilon} t}}{1-(\frac{\varepsilon K}{\gamma+\varepsilon\gamma_{s}}-\frac{K}{\gamma+\gamma_{f}})}|z_{0}-\bar{z}_{0}|,\ ~\mbox{a.s.}~\omega\in\Omega,
\end{equation*}
with $t\geq0$ and $-\gamma/\varepsilon<0$.
\end{corollary}


\section{Slow Manifolds }\label{Section5}
By the scaling $t\rightarrow\varepsilon t$, the system \eqref{3.6} can be rewritten as
\begin{eqnarray}\label{4.1}
\left\{\begin{array}{l}
d\hat{x}^\varepsilon=\varepsilon S\hat{x}^\varepsilon dt+\varepsilon g_{1}(\hat{x}^\varepsilon,\hat{y}^\varepsilon+\sigma\eta^{\frac{1}{\varepsilon}}(\theta_{\varepsilon t}\omega))dt,\\
d\hat{y}^\varepsilon=F\hat{y}^\varepsilon dt+g_{2}(\hat{x}^\varepsilon,\hat{y}^\varepsilon+\sigma\eta^{\frac{1}{\varepsilon}}(\theta_{\varepsilon t}\omega)dt.
\end{array}\right.
\end{eqnarray}
If we now replace $\eta^{\frac{1}{\varepsilon}}(\theta_{\varepsilon t}\omega)$ by $\xi(\theta_{t}\omega)$, we get the following random dynamical system
\begin{eqnarray}\label{4.3}
\left\{\begin{array}{l}
d\tilde{x}^\varepsilon=\varepsilon S\tilde{x}^\varepsilon dt+\varepsilon g_{1}(\tilde{x}^\varepsilon,\tilde{y}^\varepsilon+\sigma\xi(\theta_{t}\omega))dt,\\
d\tilde{y}^\varepsilon=F\tilde{y}^\varepsilon dt+g_{2}(\tilde{x}^\varepsilon,\tilde{y}^\varepsilon+\sigma\xi(\theta_{t}\omega))dt,
\end{array}\right.
\end{eqnarray}
with solution of which coincides with that of the system \eqref{4.1}  in distribution.

\medskip

The slow manifold of \eqref{4.3} can be constructed in a completely analogous procedure as Section \ref{section4},
so we omit the proof
and immediately state the following theorem.
\begin{theorem}\label{slowmanifold1}
Assume that $(A_{1})-(A_{3})$ hold. Given $(x_{0},y_{0})\in\mathbb{R}^{n}$, if there exists an $\delta$ such that $\varepsilon\in(0,\delta)$ and $(\tilde{x}^{\varepsilon}(0),\tilde{y}^{\varepsilon}(0))=(x_{0},y_{0})$, then the system of integral equations
\begin{equation}\label{4.6}
\left(
  \begin{array}{ccc}
  \tilde{x}^{\varepsilon}(t)\\
  \tilde{y}^{\varepsilon}(t)\\
  \end{array}
\right)
=  \left(\!\!\!
  \begin{array}{ccc}
    & e^{S\varepsilon t}x_{0}+\varepsilon\int_{0}^{t}e^{S\varepsilon(t-s)}g_{1}(\tilde{x}^\varepsilon(s),\tilde{y}^\varepsilon(s)+\sigma\xi(\theta_{s}\omega))ds\\[1ex]
    & \int_{-\infty}^{t}e^{F(t-s)}g_{2}(\tilde{x}^\varepsilon(s),\tilde{y}^\varepsilon(s)+\sigma\xi(\theta_{s}\omega))ds\\
  \end{array}
\!\!\!\right)
\end{equation}
has a unique solution $\tilde{\phi}^{\varepsilon}(t,\omega,(x_{0},y_{0}))=\big(\tilde{x}^{\varepsilon}(t,\omega,(x_{0},y_{0})),\tilde{y}^{\varepsilon}(t,\omega,(x_{0},y_{0}))\big)\in C_{-\gamma}^{-}$. Further, system \eqref{4.3} has a slow manifold
\begin{align}\nonumber
\displaystyle\tilde{\mathcal{M}}^{\varepsilon}(\omega)&=\{(x_{0},y_{0})\in\mathbb{R}^{n}:\tilde{\phi}^{\varepsilon}(\cdot,\omega,(x_{0},y_{0}))\in C_{-\gamma}^{-}\}\\\label{4.5}
   &=\{(x_{0},\tilde{h}^{\varepsilon}(\omega,x_{0})): x_{0}\in\mathbb{R}^{n_{1}}\},
\end{align}
where
\begin{equation}\label{4.9}
\displaystyle \tilde{h}^{\varepsilon}(\omega,x_{0})= \int_{-\infty}^{0}e^{-Fs}g_{2}(\tilde{x}^\varepsilon(s),\tilde{y}^\varepsilon(s)+\sigma\xi(\theta_{s}\omega))ds
\end{equation}
is a Lipschitz continuous function
and Lipschitz constant
\begin{equation}\label{4.8}
\displaystyle {\rm Lip}\tilde{h}^{\varepsilon}(\omega,\cdot)\leq- \frac{K}{(\gamma+\gamma_{f})(1 -\rho(\varepsilon))},\ \omega\in\Omega.
\end{equation}
\end{theorem}

Now, we give the relationship of $\mathcal{M}^{\varepsilon}(\omega)$ and $\mathcal{\tilde{M}}^{\varepsilon}(\omega)$ as follows.

\begin{lemma}\label{lemma8}
Assume $(A_{1})-(A_{3})$ to be valid. The slow manifold
$\mathcal{M}^{\varepsilon}(\omega)$ (see \eqref{M}) of system \eqref{3.6} is same as the slow manifold
$\mathcal{\tilde{M}}^{\varepsilon}(\omega)$ (see \eqref{4.5}) of system \eqref{4.3} in distribution. That is, for every $x_{0}\in\mathbb{R}^{n_{1}}$,
\begin{equation}\label{4.10}
\hat{h}^{\varepsilon}(\omega,x_{0})\stackrel{d}{=}\tilde{h}^{\varepsilon}(\omega,x_{0}).
\end{equation}
\end{lemma}
\begin{proof}
By the scaling $s\to\varepsilon s$ in \eqref{3.20} and the fact that the solution of system \eqref{3.6} coincides with that of system \eqref{4.3} in distribution, for every $x_{0}\in\mathbb{R}^{n_{1}}$,
\begin{align*}
\displaystyle \hat{h}^{\varepsilon}(\omega,x_{0})&=\frac{1}{\varepsilon}\int_{-\infty}^{0}e^{\frac{-Fs}{\varepsilon}}g_{2}(\hat{x}^\varepsilon(s),\hat{y}^\varepsilon(s)+\sigma\eta^{\frac{1}{\varepsilon}}(\theta_{ s}\omega))ds\\
&=\int_{-\infty}^{0}e^{-Fs}g_{2}(\hat{x}^\varepsilon(\varepsilon s),\hat{y}^\varepsilon(\varepsilon s)+\sigma\eta^{\frac{1}{\varepsilon}}(\theta_{\varepsilon s}\omega))ds\\
  &\stackrel{d}{=}\int_{-\infty}^{0}e^{-Fs}g_{2}(\tilde{x}^\varepsilon(s),\tilde{y}^\varepsilon(s)+\sigma\xi(\theta_{s}\omega))ds\\
  &=\tilde{h}^{\varepsilon}(\omega,x_{0})
\end{align*}
which completes the proof.
\end{proof}
\medskip

We are going to study the limiting case of the slow manifold for the system \eqref{3.6} as $\varepsilon\rightarrow0$
and construct an asymptotic approximation of $\mathcal{M}^{\varepsilon}(\omega)$ with sufficiently small $\varepsilon>0$ in distribution.  However, it makes also sense to study \eqref{4.3} for $\varepsilon=0$. In that
case, there exists a slow manifold.

Consider the following system
\begin{eqnarray}\label{4.11}
dx^0(t)=0,~~
dy^0(t)=Fy^0(t)dt+g_{2}(x^0(t),y^0(t)+\sigma\xi(\theta_{t}\omega))dt
\end{eqnarray}
with the initial condition $(x^{0}(0),y^{0}(0))=(x_{0},y_{0})$. As proved in Section \ref{section4}, we
also have the following result. The system \eqref{4.11} has  the following slow  manifold
\begin{equation}\label{4.13}
\mathcal{M}^{0}(\omega)=\{(x_{0},h^{0}(\omega,x_{0})): x_{0}\in\mathbb{R}^{n_{1}}\}
\end{equation}
where
\begin{equation}\label{4.14}
h^{0}(\omega,x_{0})=\int_{-\infty}^{0}e^{-Fs}g_{2}(x_{0},y^{0}(s)+\sigma\xi(\theta_{s}\omega))ds,
\end{equation}
whose Lipschitz constant ${\rm Lip}h^{0}$ satisfies
\begin{equation*}
  \displaystyle {\rm Lip}h^{0}\leq-\frac{K}{\gamma+\gamma_{f}+K},
\end{equation*}
and $y^{0}(t)$ is the unique solution in $C^{f,-}_{-\gamma}$ for integral equation
\begin{equation}\label{4.15}
y^{0}(t)=\int_{-\infty}^{t}e^{F(t-s)}g_{2}(x_{0},y^{0}(s)+\sigma\xi(\theta_{s}\omega))ds,\ t\leq0.
\end{equation}
\begin{remark}
From $0<K<-(\gamma+\gamma_{f})$, we easily have $-(\gamma+\gamma_{f}+K)>0$.
\end{remark}

As we will show, the slow manifold of the system \eqref{3.6} converges to the slow manifold
of the system \eqref{4.11} in distribution. In other words, the distribution of $\mathcal{M}^{\varepsilon}(\omega)$
converges to the distribution of $\mathcal{M}^{0}(\omega)$, as $\varepsilon$ tends to zero.
The slow manifold $\mathcal{M}^{0}(\omega)$ is called the critical manifold for
the system \eqref{3.6}.
\begin{theorem}\label{theorem5}
Assume that $(A_{1})-(A_{3})$ hold and there exists a positive constant $C$
such that $|g_{1}(x,y)|\leq C$. $\mathcal{M}^{\varepsilon}(\omega)$ converges to $\mathcal{M}^{0}(\omega)$
in distribution  as $\varepsilon\to0$. In other words, for $x_{0}\in\mathbb{R}^{n_{1}}$, $\omega\in\Omega$,
\begin{equation}\label{4.16}
\hat{h}^{\varepsilon}(\omega,x_{0})\stackrel{d}{=}h^{0}(\omega,x_{0})+\mathcal{O}(\varepsilon)\ \ \text{in}\ \mathbb{R}^{n_{2}}\ \text{as} \ \varepsilon\rightarrow0.
\end{equation}
\end{theorem}
\begin{proof}
Applying Lemma \ref{lemma8}, to prove \eqref{4.16}, we can alternatively check that if
\begin{equation}\label{4.17}
\tilde{h}^{\varepsilon}(\omega,x_{0})\stackrel{d}{\longrightarrow}h^{0}(\omega,x_{0})\ \ \text{in}\ \mathbb{R}^{n_{2}}\ \text{as} \ \varepsilon\rightarrow0.
\end{equation}

From \eqref{4.9} and \eqref{4.14}, for sufficiently small $\varepsilon$, we have
\begin{equation}\label{4.18}
\begin{array}{rl}
&|\tilde{h}^{\varepsilon}(\omega,x_{0})-h^{0}(\omega,x_{0})|\\[1ex]
=&\displaystyle\Big|\int_{-\infty}^{0}e^{-Fs}\big[g_{2}(\tilde{x}^\varepsilon(s),\tilde{y}^\varepsilon(s)+\sigma\xi(\theta_{s}\omega))-g_{2}(x_{0},y^{0}(s)+\sigma\xi(\theta_{s}\omega))\big]ds\Big|\\[2ex]
\leq&\displaystyle K\int_{-\infty}^{0}e^{-\gamma_{f}s}\big(|\tilde{x}^{\varepsilon}(s)-x_{0}|+|\tilde{y}^{\varepsilon}(s)-y^{0}(s)|\big)ds.
\end{array}
\end{equation}
According to \eqref{4.6}, for $t\leq0$, it follows that
\begin{equation}\label{4.19}
\begin{array}{rl}
|\tilde{x}^{\varepsilon}(t)-x_{0}|=&\displaystyle\Big|e^{S\varepsilon t}x_{0}-x_{0}+\varepsilon\int_{0}^{t}e^{S\varepsilon(t-s)}g_{1}(\tilde{x}^\varepsilon(s),\tilde{y}^\varepsilon(s)+\sigma\xi(\theta_{s}\omega))ds\Big|\\[1ex]
\leq&\displaystyle|e^{S\varepsilon t}x_{0}-x_{0}|+\varepsilon\int_{t}^{0}e^{\varepsilon\gamma_{s}(t-s)}|g_{1}(\tilde{x}^\varepsilon(s),\tilde{y}^\varepsilon(s)+\sigma\xi(\theta_{s}\omega))|ds\\[1ex]
\leq&\displaystyle\Big|\int_{\varepsilon t}^{0}Sx_{0}e^{Su}du\Big|+\varepsilon\cdot C\cdot\int_{t}^{0}e^{\varepsilon\gamma_{s}(t-s)}ds\\[1ex]
\leq&\displaystyle|Sx_{0}|\int_{\varepsilon t}^{0}e^{\gamma_{s}u}du+\frac{C}{\gamma_{s}}(1-e^{\varepsilon\gamma_{s}t}):=C_{1}(1-e^{\varepsilon\gamma_{s}t}),
\end{array}
\end{equation}
where $C_{1}=\frac{1}{\gamma_{s}}(|Sx_{0}|+C)$. Using \eqref{4.6} and \eqref{4.15}, it is clear that
{\small\begin{equation}\label{4.20}
\begin{array}{rl}
|\tilde{y}^{\varepsilon}(t)-y^{0}(t)|=&\displaystyle\Big|\int_{-\infty}^{t}e^{F(t-s)}\big[g_{2}(\tilde{x}^\varepsilon(s),\tilde{y}^\varepsilon(s)+\sigma\xi(\theta_{s}\omega))-g_{2}(x_{0},y^{0}(s)+\sigma\xi(\theta_{s}\omega))\big]ds\Big|\\[1ex]
\leq&\displaystyle K\int_{-\infty}^{t}e^{\gamma_{f}(t-s)}\big(|\tilde{x}^{\varepsilon}(s)-x_{0}|+|\tilde{y}^{\varepsilon}(s)-y^{0}(s)|\big)ds\\[1ex]
\leq&\displaystyle K\int_{-\infty}^{t}e^{\gamma_{f}(t-s)}\big[C_{1}(1-e^{\varepsilon\gamma_{s} s})+|\tilde{y}^{\varepsilon}(s)-y^{0}(s)|\big]ds\\[1ex]
=&\displaystyle KC_{1}(\frac{1}{\gamma_{f}-\varepsilon\gamma_{s}}e^{\varepsilon\gamma_{s}t}-\frac{1}{\gamma_{f}})+K\int_{-\infty}^{t}e^{\gamma_{f}(t-s)}|\tilde{y}^{\varepsilon}(s)-y^{0}(s)|ds.
\end{array}
\end{equation}}
Hence, we have
{\small\begin{align}\nonumber
||\tilde{y}^{\varepsilon}-y^{0}||_{C_{-\gamma}^{f,-}}&\leq\sup_{t\in (-\infty,0]}e^{\gamma t}\big[KC_{1}(\frac{1}{\gamma_{f}-\varepsilon\gamma_{s}}e^{\varepsilon\gamma_{s}t}-\frac{1}{\gamma_{f}})+K\int_{-\infty}^{t}e^{\gamma_{f}(t-s)}|\tilde{y}^{\varepsilon}(s)-y^{0}(s)|ds\big]\\\label{4.21}
  &\leq KC_{1}\sup_{t\in (-\infty,0]}q(t,\varepsilon)-\frac{K}{\gamma+\gamma_{f}}||\tilde{y}^{\varepsilon}-y^{0}||_{C_{-\gamma}^{f,-}},
\end{align}}
where
\begin{equation*}
q(t,\varepsilon)=\frac{1}{\gamma_{f}-\varepsilon\gamma_{s}}e^{(\gamma+\varepsilon\gamma_{s})t}-\frac{1}{\gamma_{f}}e^{\gamma t},\ \ t\leq0.
\end{equation*}
Since
\begin{align*}
  \frac{dq(t,\varepsilon)}{dt}&=e^{\gamma t}\left(\frac{\gamma+\varepsilon\gamma_{s}}{\gamma_{f}-\varepsilon\gamma_{s}}e^{\varepsilon\gamma_{s}t}-\frac{\gamma}{\gamma_{f}}\right)
  \geq e^{\gamma t}\left(\frac{\gamma+\varepsilon\gamma_{s}}{\gamma_{f}-\varepsilon\gamma_{s}}-\frac{\gamma}{\gamma_{f}}\right)
 \to0,\ \text{as} \ \varepsilon\to0,
\end{align*}
which implies  $q(t,\varepsilon)$ is increasing with respect to the variable $t$ for  small $\varepsilon>0$.
Then we immediately have
\begin{equation}\label{4.22}
q(t,\varepsilon)\leq q(0,\varepsilon)=\frac{1}{\gamma_{f}-\varepsilon\gamma_{s}}-\frac{1}{\gamma_{f}},\ \ t\leq0.
\end{equation}
According to \eqref{4.21} and \eqref{4.22}, we obtain
\begin{equation*}
||\tilde{y}^{\varepsilon}-y^{0}||_{C_{-\gamma}^{f,-}}\leq C_{2}\big(\frac{1}{\gamma_{f}-\varepsilon\gamma_{s}}-\frac{1}{\gamma_{f}}\big), \ \ \text{with}\ \ C_{2}=\frac{KC_{1}}{1+\frac{K}{\gamma+\gamma_{f}}}.
\end{equation*}
Hence
\begin{equation}\label{4.23}
|\tilde{y}^{\varepsilon}(t)-y^{0}(t)|\leq C_{2}e^{-\gamma t}\big(\frac{1}{\gamma_{f}-\varepsilon\gamma_{s}}-\frac{1}{\gamma_{f}}\big),\ \ t\leq0.
\end{equation}
It follows from \eqref{4.18}, \eqref{4.19} and \eqref{4.23} that
\begin{align*}
&|\tilde{h}^{\varepsilon}(\omega,x_{0})-h^{0}(\omega,x_{0})|\\
  &\leq\displaystyle K\big[C_{1}\int_{-\infty}^{0}e^{-\gamma_{f}s}(1-e^{\varepsilon\gamma_{s}s})ds+ C_{2}\big(\frac{1}{\gamma_{f}-\varepsilon\gamma_{s}}-\frac{1}{\gamma_{f}}\big)\int_{-\infty}^{0}e^{-(\gamma+\gamma_{f})s}ds\big]\\
  &=C_{3}\big(\frac{1}{\gamma_{f}-\varepsilon\gamma_{s}}-\frac{1}{\gamma_{f}}\big)\to0,\ \ \text{as}\ \ \varepsilon\to0,
\end{align*}
where $C_{3}=K\left(C_{1}-\frac{C_{2}}{\gamma+\gamma_{f}}\right)=\frac{K(\gamma+\gamma_{f})(|Sx_{0}|+C)}{\gamma_{s}(\gamma+\gamma_{f}+K)}$.
This completes the proof.
\end{proof}
\begin{theorem}\label{theorem6}
Assume the hypotheses of Theorem \ref{theorem5} to be valid. Then there exists a $\delta$ such that if $\varepsilon\in(0,\delta)$, the slow manifold of system \eqref{3.6} can be approximated in distribution as
\begin{align}\label{4.24}
\mathcal{M}^{\varepsilon}(\omega)
                                 &\stackrel{d}{=}\{(x_{0},h^{0}(\omega,x_{0})+\varepsilon h^{1}(\omega,x_{0})+\mathcal{O}(\varepsilon^{2})): X_{0}\in\mathbb{R}^{n_{1}}\}
\end{align}
where $h^{0}(\omega,x_{0})$ is defined in \eqref{4.14},
\begin{equation}\label{4.35}
h^{1}(\omega,x_{0})=\int_{-\infty}^{0}e^{-Fs}[x^{1}(s)g_{2,x}(x_{0},y^{0}(s)+\sigma\xi(\theta_{s}\omega))
                              +y^{1}(s)g_{2,y}(x_{0},y^{0}(s)+\sigma\xi(\theta_{s}\omega))]ds,
\end{equation}
and $(x^{1}(t),y^{1}(t))$ are given by \eqref{4.35} and \eqref{4.37}.
\end{theorem}
\begin{proof}
Applying Lemma \ref{lemma8}, we can alternatively prove
\begin{equation}\label{4.25}
\tilde{h}^{\varepsilon}(\omega,x_{0})=h^{0}(\omega,x_{0})+\varepsilon h^{1}(\omega,x_{0})+\mathcal{O}(\varepsilon^{2})
\end{equation}
For the system \eqref{4.3}, we write
\begin{eqnarray}\label{4.26}
\begin{array}{l}
\displaystyle \tilde{x}^{\varepsilon}(t)=\tilde{x}^{0}(t)+\varepsilon x^{1}(t)+\mathcal{O}(\varepsilon^{2}),~~\tilde{x}^{\varepsilon}(0)=x_{0},\\[1ex]
\displaystyle \tilde{y}^{\varepsilon}(t)=\tilde{y}^{0}(t)+\varepsilon y^{1}(t)+\mathcal{O}(\varepsilon^{2}),~~\tilde{y}^{\varepsilon}(0)=y_{0},
\end{array}
\end{eqnarray}
where $\tilde{x}^{0}(t), \tilde{y}^{0}(t), x^{1}(t)$ and $y^{1}(t)$ will be determined in the below.
The Taylor expansions of $g_{i}(\tilde{x}^\varepsilon(t),\tilde{y}^\varepsilon(t)+\sigma\xi(\theta_{t}\omega)), i=1, 2$  at point $(\tilde{x}^0(t),\tilde{y}^0(t)+\sigma\xi(\theta_{t}\omega))$ are as follows.
\begin{equation}\label{4.28}
\begin{array}{rl}
&g_{i}(\tilde{x}^\varepsilon(t),\tilde{y}^\varepsilon(t)+\sigma\xi(\theta_{t}\omega))\\[1ex]
=&g_{i}(\tilde{x}^0(t),\tilde{y}^0(t)+\sigma\xi(\theta_{t}\omega))+(\tilde{x}^\varepsilon(t)-\tilde{x}^0(t))g_{i,x}(\tilde{x}^0(t),\tilde{y}^0(t)+\sigma\xi(\theta_{t}\omega))\\[1ex]
&+(\tilde{y}^\varepsilon(t)-\tilde{y}^0(t))g_{i,y}(\tilde{x}^0(t),\tilde{y}^0(t)+\sigma\xi(\theta_{t}\omega))+\mathcal{O}(\varepsilon^{2})\\[1ex]
=&g_{i}(\tilde{x}^0(t),\tilde{y}^0(t)+\sigma\xi(\theta_{t}\omega))+\varepsilon x^{1}(t)g_{i,x}(\tilde{x}^0(t),\tilde{y}^0(t)+\sigma\xi(\theta_{t}\omega))\\[1ex]
&+\varepsilon y^{1}(t)g_{i,y}(\tilde{x}^0(t),\tilde{y}^0(t)+\sigma\xi(\theta_{t}\omega))+\mathcal{O}(\varepsilon^{2}),
\end{array}
\end{equation}
where $g_{i,x}(x,y)$ and $g_{i,y}(x,y)$ denote the partial derivative of $g_{i}(x,y)$ with respect to the  variables $x$ and $y$ respectively.

Substituting \eqref{4.26} into \eqref{4.3}, equating the terms with the same power of $\varepsilon$, we deduce that
\begin{align}\label{4.30}
  d\tilde{x}^{0}(t)&=0\\\label{4.31}
  dx^{1}(t)&=S\tilde{x}^{0}(t)dt+g_{1}(\tilde{x}^0(t),\tilde{y}^0(t)+\sigma\xi(\theta_{t}\omega))dt
\end{align}
and
\begin{align}\label{4.32}
  d\tilde{y}^{0}(t)&=F\tilde{y}^{0}(t)dt+g_{2}(\tilde{x}^0(t),\tilde{y}^0(t)+\sigma\xi(\theta_{t}\omega))dt\\ \nonumber
    dy^{1}(t)&=x^{1}(t)g_{2,x}(\tilde{x}^0(t),\tilde{y}^0(t)+\sigma\xi(\theta_{t}\omega))dt\\\label{4.33}
&\ \ +y^{1}(t)[F+g_{2,y}(\tilde{x}^0(t),\tilde{y}^0(t)+\sigma\xi(\theta_{t}\omega))]dt.
\end{align}
Comparing \eqref{4.11}    with \eqref{4.30} and \eqref{4.32}, we immediately have
\begin{equation}\label{4.34}
x^{0}(t)=\tilde{x}^{0}(t),\ \ y^{0}(t)=\tilde{y}^{0}(t),
\end{equation}
which implies that the system \eqref{4.11} essentially is the system \eqref{3.6} scaled by $\varepsilon t$ with zero singular perturbation parameter, i.e., the system \eqref{4.3} with $\varepsilon=0$.
From \eqref{4.31} and $\tilde{x}^{0}(0)=x_{0}$, we get
\begin{equation}\label{4.35}
x^{1}(t)=Sx_{0}t+\int_{0}^{t}g_{1}(x_{0},y^{0}(s)+\sigma\xi(\theta_{s}\omega))dt.
\end{equation}
According to \eqref{4.32}, \eqref{4.34} and $\tilde{y}^{0}(0)=h^{0}(\omega,x_{0})$, we obtain
\begin{equation}\label{4.36}
\tilde{y}^{0}(t)=e^{Ft}h^{0}(\omega,x_{0})+\int_{0}^{t}e^{-F(s-t)}g_{_{2}}(x_{0},\tilde{y}^{0}(s)+\sigma\xi(\theta_{s}\omega))ds.
\end{equation}
By \eqref{4.33}-\eqref{4.35} and $\tilde{y}^{1}(0)=h^{1}(\omega,x_{0})$, we have
\begin{align}\nonumber
y^{1}(t)&=e^{Ft+\int_{0}^{t}g_{2,y}(x_{0},\tilde{y}^{0}(s)+\sigma\xi(\theta_{s}\omega))ds}h^{1}(\omega,x_{0})+\int_{0}^{t}e^{-F(s-t)+\int_{s}^{t}g_{2,y}(x_{0},\tilde{y}^{0}(u)+\sigma\xi(\theta_{u}\omega))du} \\\label{4.37}
   &\cdot g_{2,x}(x_{0},\tilde{y}^{0}(s)+\sigma\xi(\theta_{s}\omega))\big[Sx_{0}s+\int_{0}^{s}g_{1}(x_{0},y^{0}(u)+\sigma\xi(\theta_{u}\omega))du\big]ds
\end{align}
It follows from \eqref{4.9}, \eqref{4.28} and \eqref{4.34} that
{\small\begin{align*}
   \tilde{h}^{\varepsilon}(\omega,x_{0})=&\displaystyle\int_{-\infty}^{0}e^{-Fs}g_{2}(\tilde{x}^\varepsilon(s),\tilde{y}^\varepsilon(s)+\sigma\xi(\theta_{s}\omega))ds \\
=&\displaystyle\int_{-\infty}^{0}e^{-Fs}[g_{2}(\tilde{x}^0(s),\tilde{y}^0(s)+\sigma\xi(\theta_{s}\omega))+\varepsilon x^{1}(s)g_{2,x}(\tilde{x}^0(s),\tilde{y}^0(s)+\sigma\xi(\theta_{s}\omega)) \\[1ex]
   &\displaystyle+\varepsilon y^{1}(s)g_{2,y}(\tilde{x}^0(s),\tilde{y}^0(t)+\sigma\xi(\theta_{s}\omega))]ds+\mathcal{O}(\varepsilon^{2}) \\[1ex]
=&\displaystyle\int_{-\infty}^{0}e^{-Fs}[g_{2}(x_{0},y^0(s)+\sigma\xi(\theta_{s}\omega))ds+\varepsilon\int_{-\infty}^{0}e^{-Fs}[x^{1}(s)g_{2,x}(x_{0},y^{0}(s)+\sigma\xi(\theta_{s}\omega))\\
   &+y^{1}(s)g_{2,y}(x_{0},y^{0}(s)+\sigma\xi(\theta_{s}\omega))]ds+\mathcal{O}(\varepsilon^{2})\\[1ex]
=&\displaystyle h^{0}(\omega,x_{0})+\varepsilon h^{1}(\omega,x_{0})+\mathcal{O}(\varepsilon^{2}),
\end{align*}}
which conclude the proof.
\end{proof}

\section{Examples }\label{Section6}
Now we present three examples from biological sciences to illustrate our analytical results.
\begin{example}
Consider a two dimensional model of FitzHugh-Nagumo system \cite{Zi08}
\begin{eqnarray}\label{6.1}
\left\{\begin{array}{l}
 dx^\varepsilon=x^\varepsilon dt+\frac{1}{3}\sin y^\varepsilon dt,~~x^\varepsilon\in\mathbb{R}\\
dy^\varepsilon=-\frac{1}{\varepsilon}y^\varepsilon dt+\frac{1}{6\varepsilon}(x^\varepsilon)^{2} dt+\sigma\varepsilon^{-\frac{1}{\alpha}}dL_t^{\alpha},~~y^\varepsilon\in\mathbb{R}
 \end{array}\right.
\end{eqnarray}
where $x^\varepsilon$
is the ``slow" component, $y^\varepsilon$ is the ``fast" component, $\gamma_{s}=1$, $\gamma_{f}=-1$, $K<1$, $g_{1}(x^\varepsilon, y^\varepsilon)=\frac{1}{3}\sin y^\varepsilon$, $g_{2}(x^\varepsilon, y^\varepsilon)=\frac{1}{6}(x^\varepsilon)^{2}$.

The scaling $t\rightarrow\varepsilon t$ in \eqref{6.1} yields
\begin{eqnarray}\label{xin1}
\left\{\begin{array}{l}
 dx^\varepsilon=\varepsilon x^\varepsilon dt+\frac{1}{3}\varepsilon\sin y^\varepsilon dt,\\[1ex]
dy^\varepsilon=-y^\varepsilon dt+\frac{1}{6}(x^\varepsilon)^{2} dt+\sigma dL_t^{\alpha}.
 \end{array}\right.
\end{eqnarray}
We can convert this two dimensional SDE system to the following system
\begin{eqnarray}\label{4.39}
\left\{\begin{array}{l}
d\tilde{x}^\varepsilon=\varepsilon\tilde{x}^\varepsilon dt+\frac{\varepsilon}{3}\sin(\tilde{y}^\varepsilon+\sigma\xi(\theta_{t}\omega))dt, \\[1ex]
d\tilde{y}^\varepsilon=-\tilde{y}^\varepsilon+\frac{1}{6}(\tilde{x}^\varepsilon)^{2},
 \end{array}\right.
\end{eqnarray}
where $\xi(\theta_{t}\omega)=\int_{-\infty}^{t} e^{(t-s)}dL_{s}^{\alpha}$.

Denote $\tilde{x}^\varepsilon(0)=x_{0}$, we get
$h^{0}(\omega,x_{0})=\frac{x^{2}_{0}}{6}$ and
\begin{equation*}
h^{1}(\omega,x_{0})=-\frac{x^{2}_{0}}{3}+\frac{x_{0}}{9}\int_{-\infty}^{0}e^{t}\Big[\int_{0}^{t}\sin\big(\frac{ x^{2}_{0}}{6}+\sigma\int_{-\infty}^{s}e^{s-r}dL_{r}^{\alpha}\big)ds\Big]dt.
\end{equation*}
This produces an approximated slow manifold $\tilde{\mathcal{M}}^{\varepsilon}(\omega)=\{(x_{0},\tilde{h}^{\varepsilon}(\omega,x_{0}): x_{0}\in[0,\pi]\}$ of system \eqref{4.39}, where $\tilde{h}^{\varepsilon}(\omega,x_{0})=h^{0}(\omega,x_{0})+\varepsilon h^{1}(\omega,x_{0})+\mathcal{O}(\varepsilon^{2})$. Then we obtain
\begin{eqnarray}\label{rd1}
\left\{\begin{array}{l}
 d\bar{x}^\varepsilon=\varepsilon \bar{x}^\varepsilon dt+\frac{1}{3}\varepsilon\sin \bar{y}^\varepsilon dt,\\[1ex]
\bar{y}^\varepsilon=\tilde{h}^{\varepsilon}(\theta_{t}\omega,\bar{x}^\varepsilon),
 \end{array}\right.
\end{eqnarray}
which is the reduction system of \eqref{xin1}.
\begin{figure}[H]
\centering
\includegraphics[width=2.92in]{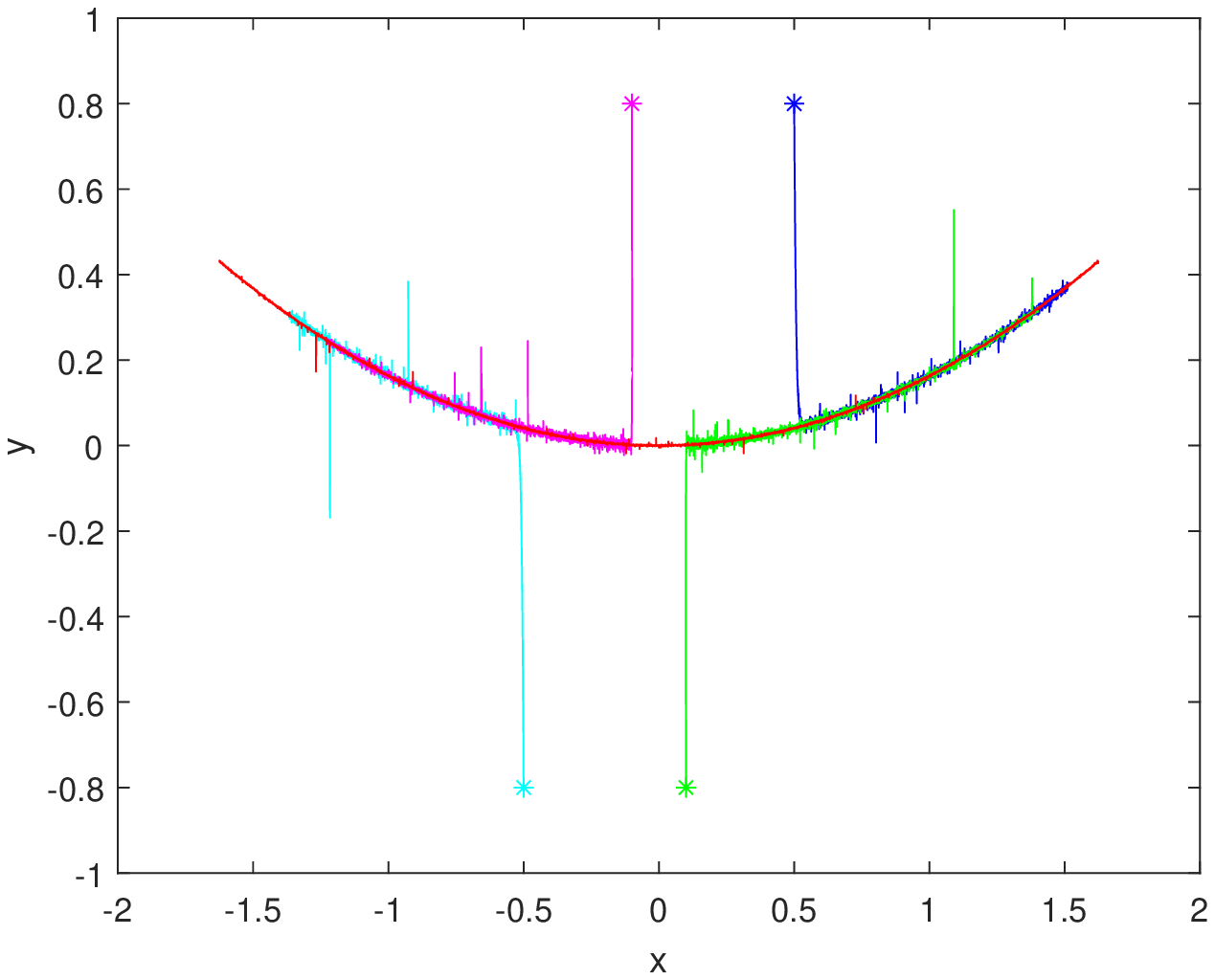}
\includegraphics[width=2.92in]{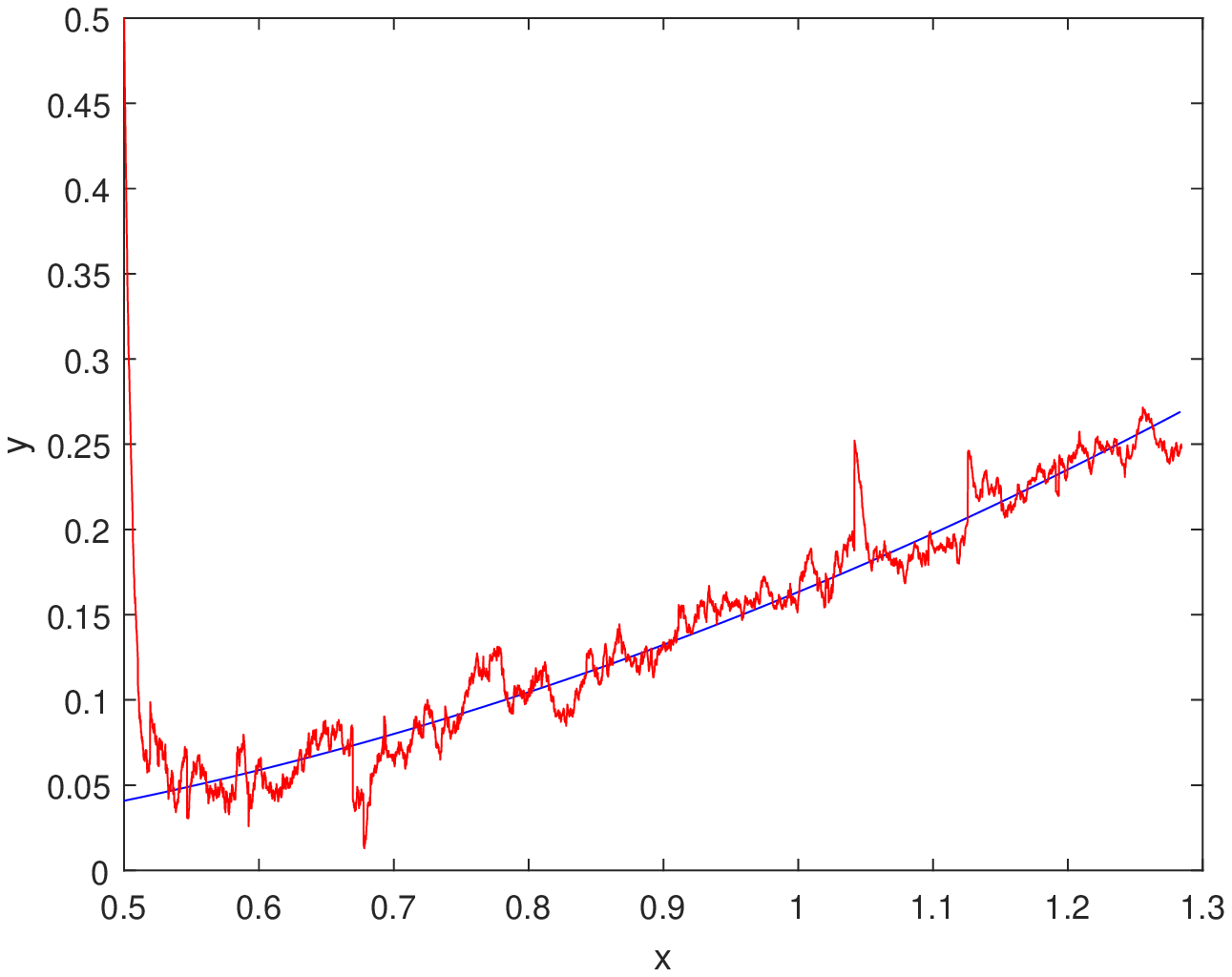}
\caption{(Online color)Four orbits of slow-fast system \eqref{xin1} and its slow manifold expansion $\tilde{h}^{\varepsilon}(\omega,x_{0})$; and orbits for system \eqref{xin1} and slow manifold reduced system \eqref{rd1}: $\varepsilon=0.01$, $\sigma=0.05$ and $\alpha=1.8$.}
\end{figure}
\end{example}

\begin{example}
Consider a three dimensional model of FitzHugh-Nagumo system
\begin{eqnarray}\label{6.3}
\left\{\begin{array}{l}
 dx_{1}^\varepsilon=\frac{1}{2}x_{1}^\varepsilon dt+(-\frac{(x_{1}^\varepsilon)^{3}+x_{1}^\varepsilon x_{2}^\varepsilon}{20}+\frac{y^\varepsilon}{3})dt,~~x_{1}^{\varepsilon}\in\mathbb{R} \\
 dx_{2}^\varepsilon=\frac{1}{3}x_{2}^\varepsilon dt+(\frac{1}{2}\sin x_{1}^\varepsilon\cos x_{2}^\varepsilon+\frac{(y^\varepsilon)^{2}}{8})dt,~~x_{2}^\varepsilon\in\mathbb{R} \\
dy^\varepsilon=-\frac{1}{\varepsilon}y^\varepsilon dt-\frac{1}{10\varepsilon}x_{1}^\varepsilon x_{2}^\varepsilon dt+\sigma\varepsilon^{-\frac{1}{\alpha}}dL_t^{\alpha},~~y^\varepsilon\in\mathbb{R}
 \end{array}\right.
\end{eqnarray}
where $(x_{1}^\varepsilon,x_{2}^\varepsilon)$
is the ``slow" component, $y^\varepsilon$ is the ``fast" component,$\gamma_{s}=\frac{1}{3}$, $\gamma_{f}=-1$, $K<1$, $g_{1}(x_{1}^\varepsilon, x_{2}^\varepsilon, y^\varepsilon)=-\frac{(x_{1}^\varepsilon)^{3}+x_{1}^\varepsilon x_{2}^\varepsilon}{20}+\frac{y^\varepsilon}{3}$, $g_{2}(x_{1}^\varepsilon, x_{2}^\varepsilon, y^\varepsilon)=\frac{1}{2}\sin x_{1}^\varepsilon\cos x_{2}^\varepsilon+\frac{(y^\varepsilon)^{2}}{8}$, $g(x_{1}^\varepsilon, x_{2}^\varepsilon, y^\varepsilon)=-\frac{1}{10}x_{1}^\varepsilon x_{2}^\varepsilon$.

We now scale the time $t\rightarrow\varepsilon t$, \eqref{6.3} can be rewriten as
\begin{eqnarray}\label{xin2}
\left\{\begin{array}{l}
 dx_{1}^\varepsilon=\frac{1}{2}\varepsilon x_{1}^\varepsilon dt+\varepsilon(-\frac{(x_{1}^\varepsilon)^{3}+x_{1}^\varepsilon x_{2}^\varepsilon}{20}+\frac{y^\varepsilon}{3})dt,\\
 dx_{2}^\varepsilon=\frac{1}{3}\varepsilon x_{2}^\varepsilon dt+\varepsilon(\frac{1}{2}\sin x_{1}^\varepsilon\cos x_{2}^\varepsilon+\frac{(y^\varepsilon)^{2}}{8})dt,\\
dy^\varepsilon=-y^\varepsilon dt-\frac{1}{10}x_{1}^\varepsilon x_{2}^\varepsilon dt+\sigma dL_t^{\alpha}.
 \end{array}\right.
\end{eqnarray}
The corresponding system of random differential equations
\begin{eqnarray}\label{6.4}
\left\{\begin{array}{l}
 d\tilde{x}_{1}^\varepsilon=\frac{1}{2}\varepsilon\tilde{x}_{1}^\varepsilon dt+\varepsilon[-\frac{(\tilde{x}_{1}^\varepsilon)^{3}+\tilde{x}_{1}^\varepsilon\tilde{x}_{2}^\varepsilon}{20}+\frac{1}{3}(\tilde{y}^\varepsilon+\sigma\xi(\theta_{t}\omega))]dt \\[1ex]
 d\tilde{x}_{2}^\varepsilon=\frac{1}{3}\varepsilon\tilde{x}_{2}^\varepsilon dt+\varepsilon\big[\frac{1}{2}\sin \tilde{x}_{1}^\varepsilon\cos \tilde{x}_{2}^\varepsilon+\frac{1}{8}(\tilde{y}^\varepsilon+\sigma\xi(\theta_{t}\omega))^{2}\big]dt\\[1ex]
d\tilde{y}^\varepsilon=-\tilde{y}^\varepsilon dt-\frac{1}{10}\tilde{x}_{1}^\varepsilon\tilde{x}_{2}^\varepsilon
  dt
 \end{array}\right.
\end{eqnarray}
where $\xi(\theta_{t}\omega)=\int_{-\infty}^{t} e^{(t-s)}dL_{s}^{\alpha}$.

Denote $(\tilde{x}_{1}^{\varepsilon}(0),\tilde{x}_{2}^{\varepsilon}(0))=(x_{0},x^{'}_{0})$, we get
$h^{0}(\omega,(x_{0},x^{'}_{0}))=-\frac{1}{10}x_{0}x_{0}^{'}$
and
\begin{align*}
&h^{1}(\omega,(x_{0},x^{'}_{0}))\\&=(\frac{10x_{0}-x_{0}^{3}}{20}-\frac{x_{0}x_{0}^{'}}{12})\frac{x_{0}^{'}}{10}+\frac{\sigma(3x_{0}^{2} -40)x_{0}^{'}}{1200}\int_{-\infty}^{0}e^{t}\Big[\int_{0}^{t}\int_{-\infty}^{s}e^{s-r}dL_{r}^{\alpha}ds\Big]dt\\
  &+(\frac{x^{'}_{0}}{3}+\frac{1}{2}\sin x_{0}\cos x_{0}^{'}+\frac{(x_{0}x_{0}^{'})^{2}}{800})\frac{x_{0}}{10}-\frac{\sigma^{2} x_{0}}{80}\int_{-\infty}^{0}e^{t}\Big[\int_{0}^{t}\left(\int_{-\infty}^{s}e^{s-r}dL_{r}^{\alpha}\right)^{2}ds\Big]dt.
\end{align*}
We get an approximated slow manifold $\tilde{\mathcal{M}}^{\varepsilon}(\omega)=\{((x_{0},x^{'}_{0}),\tilde{h}(\omega,(x_{0},x^{'}_{0})): (x_{0},x^{'}_{0})\in\mathbb{R}^{2}\}$ of system \eqref{6.4}, where $\tilde{h}^{\varepsilon}(\omega,(x_{0},x^{'}_{0}))=h^{0}(\omega,(x_{0},x^{'}_{0}))+\varepsilon h^{1}(\omega,(x_{0},x^{'}_{0}))+\mathcal{O}(\varepsilon^{2})$. The reduced system of \eqref{xin2} given by
\begin{eqnarray}\label{rd2}
\left\{\begin{array}{l}
 d\bar{x}_{1}^\varepsilon=\frac{1}{2}\varepsilon \bar{x}_{1}^\varepsilon dt+\varepsilon(-\frac{(\bar{x}_{1}^\varepsilon)^{3}+\bar{x}_{1}^\varepsilon \bar{x}_{2}^\varepsilon}{20}+\frac{\bar{y}^\varepsilon}{3})dt,\\
 d\bar{x}_{2}^\varepsilon=\frac{1}{3}\varepsilon \bar{x}_{2}^\varepsilon dt+\varepsilon(\frac{1}{2}\sin \bar{x}_{1}^\varepsilon\cos \bar{x}_{2}^\varepsilon+\frac{(\bar{y}^\varepsilon)^{2}}{8})dt,\\
\bar{y}^\varepsilon=\tilde{h}^{\varepsilon}(\theta_{t}\omega,(\bar{x}_{1}^\varepsilon,\bar{x}_{2}^\varepsilon)).
 \end{array}\right.
\end{eqnarray}
\begin{figure}[H]
\centering
\includegraphics[width=3.62in]{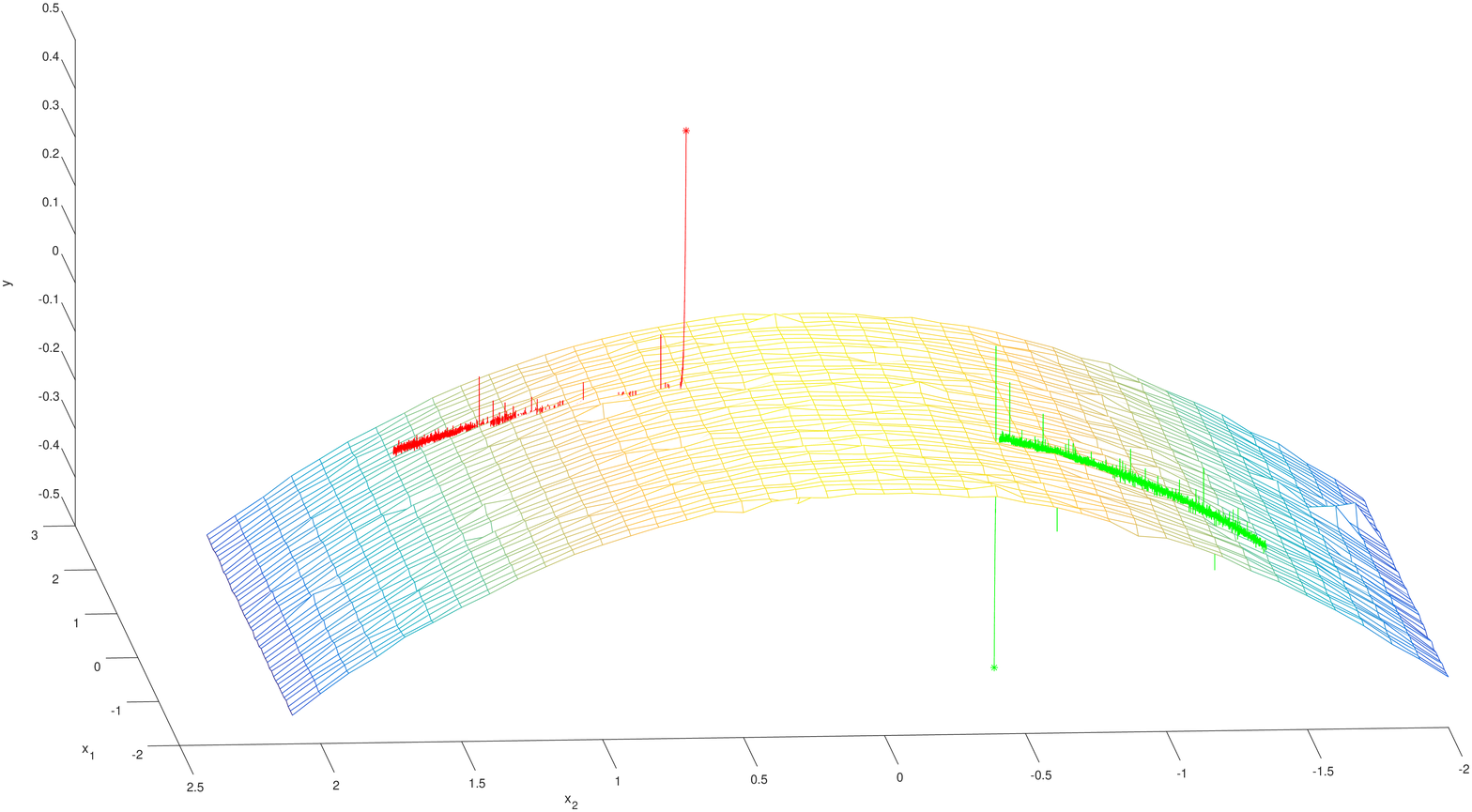}
\includegraphics[width=2.22in]{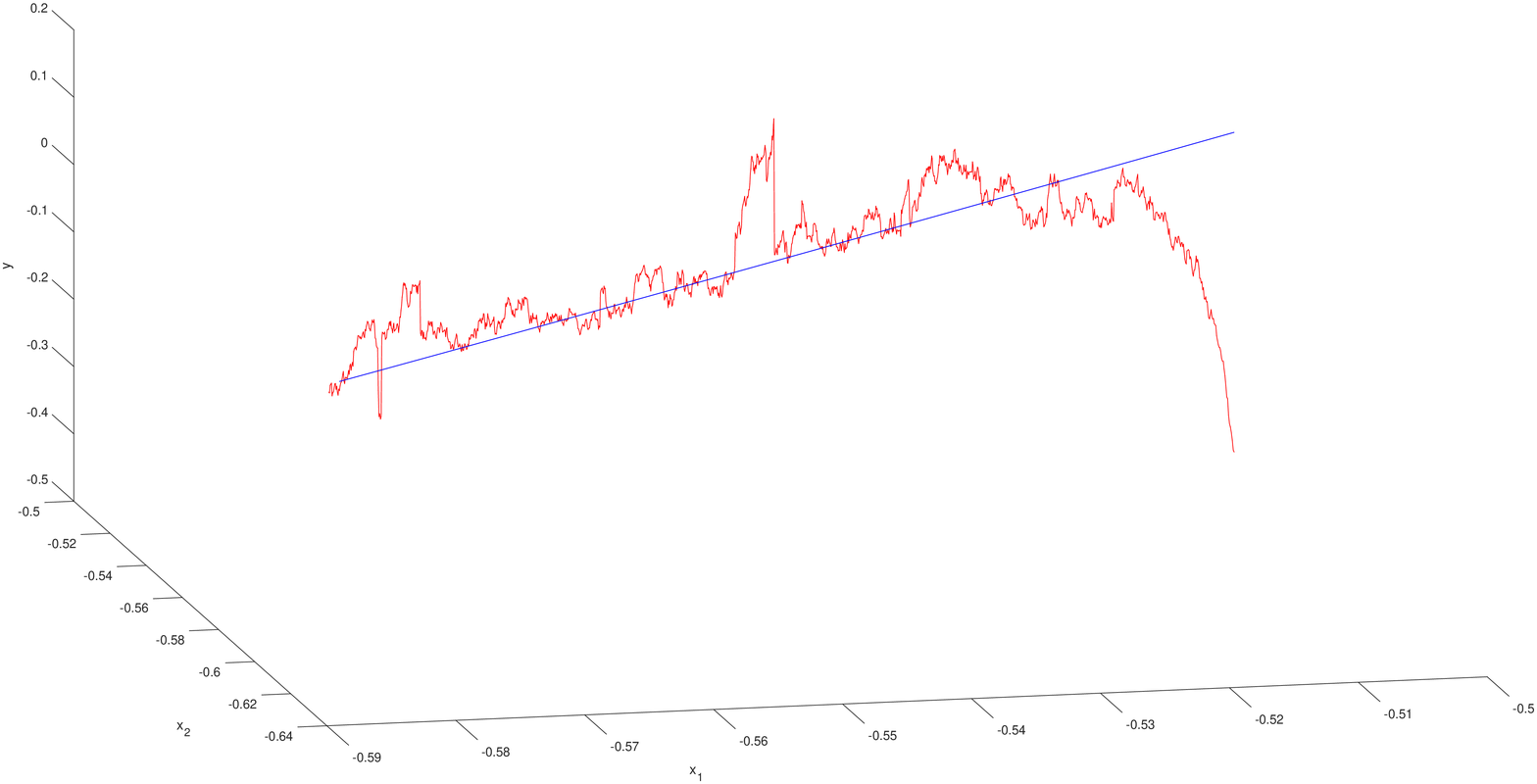}
\caption{(Online color)Two orbits of slow-fast system \eqref{xin2} and its slow manifold expansion $\tilde{h}^{\varepsilon}(\omega,(x_{0},x^{'}_{0}))$; and orbits for system \eqref{xin2} and slow manifold reduced system \eqref{rd2}: $\varepsilon=0.01$, $\alpha=1.8$, $\sigma=0.1$ (left) and $\sigma=0.05$ (right).}
\end{figure}
\end{example}
\begin{example}
Consider a three dimensional model of FitzHugh-Nagumo system
\begin{eqnarray}\label{6.5}
\left\{\begin{array}{l}
 dx^{\varepsilon}=\frac{1}{3}x^{\varepsilon}dt+\frac{x^{\varepsilon}-(x^{\varepsilon})^{3}+\sin y_{1}^{\varepsilon}\cos y_{2}^{\varepsilon}}{50}dt,~~x\in\mathbb{R} \\
 dy_{1}^{\varepsilon}=-\frac{1}{\varepsilon}y_{1}^{\varepsilon}dt+\frac{1}{5\varepsilon}sin x^{\varepsilon}
  dt+\sigma\varepsilon^{-\frac{1}{\alpha_{1}}}dL_t^{\alpha_{1}},~~y_{1}^{\varepsilon}\in\mathbb{R}\\
dy_{2}^{\varepsilon}=-\frac{1}{\varepsilon}y_{2}^{\varepsilon}dt-\frac{1}{16\varepsilon} (x^{\varepsilon})^{2}dt+\sigma\varepsilon^{-\frac{1}{\alpha_{2}}}dL_t^{\alpha_{2}},~~y_{2}^{\varepsilon}\in\mathbb{R} \\
 \end{array}\right.
\end{eqnarray}
where $x^\varepsilon$
is the ``slow" component, $(y_{1}^\varepsilon,y_{2}^\varepsilon)$ is the ``fast" component, $\gamma_{s}=\frac{1}{3}$, $\gamma_{f}=-1$, $K<1$, $g(x^\varepsilon, y_{1}^\varepsilon, y_{2}^\varepsilon)=\frac{x^\varepsilon-(x^\varepsilon)^{3}+\sin y_{1}^\varepsilon\cos y_{2}^\varepsilon}{50}$, $g_{1}(x^\varepsilon, y_{1}^\varepsilon, y_{2}^\varepsilon)=\frac{1}{5}sin x^\varepsilon$, $g_{2}(x^\varepsilon, y_{1}^\varepsilon, y_{2}^\varepsilon)=-\frac{(x ^{\varepsilon})^{2}}{16}$, and $L_t^{\alpha_{1}}$, $L_t^{\alpha_{2}}$ are independent two-sided $\alpha$-stable L\'evy motion on $\mathbb{R}$ with $1<\alpha<2$.

By using the scaling $t\rightarrow\varepsilon t$, we have
\begin{eqnarray}\label{xin3}
\left\{\begin{array}{l}
 dx^{\varepsilon}=\frac{1}{3}\varepsilon x^{\varepsilon}dt+\varepsilon\frac{x^{\varepsilon}-(x^{\varepsilon})^{3}+\sin y_{1}^{\varepsilon}\cos y_{2}^{\varepsilon}}{50}dt,\\
 dy_{1}^{\varepsilon}=-y_{1}^{\varepsilon}dt+\frac{1}{5}sin x^{\varepsilon}
  dt+\sigma dL_t^{\alpha_{1}},\\[0.3mm]
dy_{2}^{\varepsilon}=-y_{2}^{\varepsilon}dt-\frac{1}{16} (x^{\varepsilon})^{2}dt+\sigma dL_t^{\alpha_{2}},\\
 \end{array}\right.
\end{eqnarray}
which can be transformed into
\begin{eqnarray}\label{6.6}
\left\{\begin{array}{l}
 d\tilde{x}^\varepsilon=\frac{1}{3}\varepsilon \tilde{x}^\varepsilon dt+\varepsilon\frac{\tilde{x}^\varepsilon-(\tilde{x}^\varepsilon)^{3}+\sin( \tilde{y}_{1}^\varepsilon+\sigma\xi_{1}(\theta_{t}\omega))\cos(\tilde{y}_{2}^\varepsilon+\sigma\xi_{2}(\theta_{t}\omega))}{50}dt \\[0.5mm]
 d\tilde{y}_{1}^\varepsilon=-\tilde{y}_{1}^\varepsilon dt+\frac{1}{5}sin \tilde{x}^\varepsilon dt\\[0.5mm]
d\tilde{y}_{2}^\varepsilon=-\tilde{y}_{2}^\varepsilon dt-\frac{1}{16}(\tilde{x}^\varepsilon)^{2}dt\\
 \end{array}\right.
\end{eqnarray}
where
$\xi_{i}(\theta_{t}\omega)=\int_{-\infty}^{t} e^{(t-s)}dL_{s}^{\alpha_{i}}, i=1, 2.$

Denote $\tilde{x}(0)=x_{0}$, we get
\begin{equation*}
h^{0}(\omega,x_{0})
= \left(
  \begin{array}{ccc}
  \frac{1}{5}sin x_{0}\\[1mm]
  -\frac{1}{16}x^{2}_{0}\\
  \end{array}
\right).
\end{equation*}
and
{\small\begin{align*}
&h^{1}(\omega,x_{0})\\&= \left(
  \begin{array}{ccc}
  \frac{1}{5}\cos x_{0}\big[\frac{x_{0}^{3}}{50}-\frac{53x_{0}}{150}+\int_{-\infty}^{0}e^{t}\int_{0}^{t}\frac{\sin( \frac{1}{5}sin x_{0}+\sigma\int_{-\infty}^{s}e^{s-r}dL_{r}^{\alpha_{1}}ds)\cos(-\frac{1}{16}x^{2}_{0}+\sigma\int_{-\infty}^{s}e^{s-r}dL_{r}^{\alpha_{2}})}{50}dsdt\big]\\[1mm]
  -\frac{1}{8}x_{0}\big[\frac{x_{0}^{3}}{50}-\frac{53x_{0}}{150}+\int_{-\infty}^{0}e^{t}\int_{0}^{t}\frac{\sin( \frac{1}{5}sin x_{0}+\sigma\int_{-\infty}^{s}e^{s-r}dL_{r}^{\alpha_{1}}ds)\cos(-\frac{1}{16} x^{2}_{0}+\sigma\int_{-\infty}^{s}e^{s-r}dL_{r}^{\alpha_{2}})}{50}dsdt\big]\\
  \end{array}
\right).
\end{align*}}
 Then we obtain an approximated random slow manifold $\tilde{\mathcal{M}}(\omega)=\{(x_{0},\tilde{h}(\omega,x_{0}): x_{0}\in\mathbb{R}^{2}\}$ of system \eqref{6.6}, where $\tilde{h}(\omega,x_{0})=(\tilde{h}_{1}(\omega,x_{0}),\tilde{h}_{2}(\omega,x_{0}))=h^{0}(\omega,x_{0})+\varepsilon h^{1}(\omega,x_{0})+\mathcal{O}(\varepsilon^{2})$. Moreover,
 \begin{eqnarray}\label{rd3}
\left\{\begin{array}{l}
 d\bar{x}^{\varepsilon}=\frac{1}{3}\varepsilon \bar{x}^{\varepsilon}dt+\varepsilon\frac{\bar{x}^{\varepsilon}-(\bar{x}^{\varepsilon})^{3}+\sin \bar{y}_{1}^{\varepsilon}\cos \bar{y}_{2}^{\varepsilon}}{50}dt,\\
 \bar{y}_{1}^{\varepsilon}=\tilde{h}_{1}(\theta_{t}\omega,\bar{x}^{\varepsilon}),\\
\bar{y}_{2}^{\varepsilon}=\tilde{h}_{2}(\theta_{t}\omega,\bar{x}^{\varepsilon}),\\
 \end{array}\right.
\end{eqnarray}
is the reduction system of \eqref{xin3}.
\begin{figure}[H]
\centering
\includegraphics[width=2.92in]{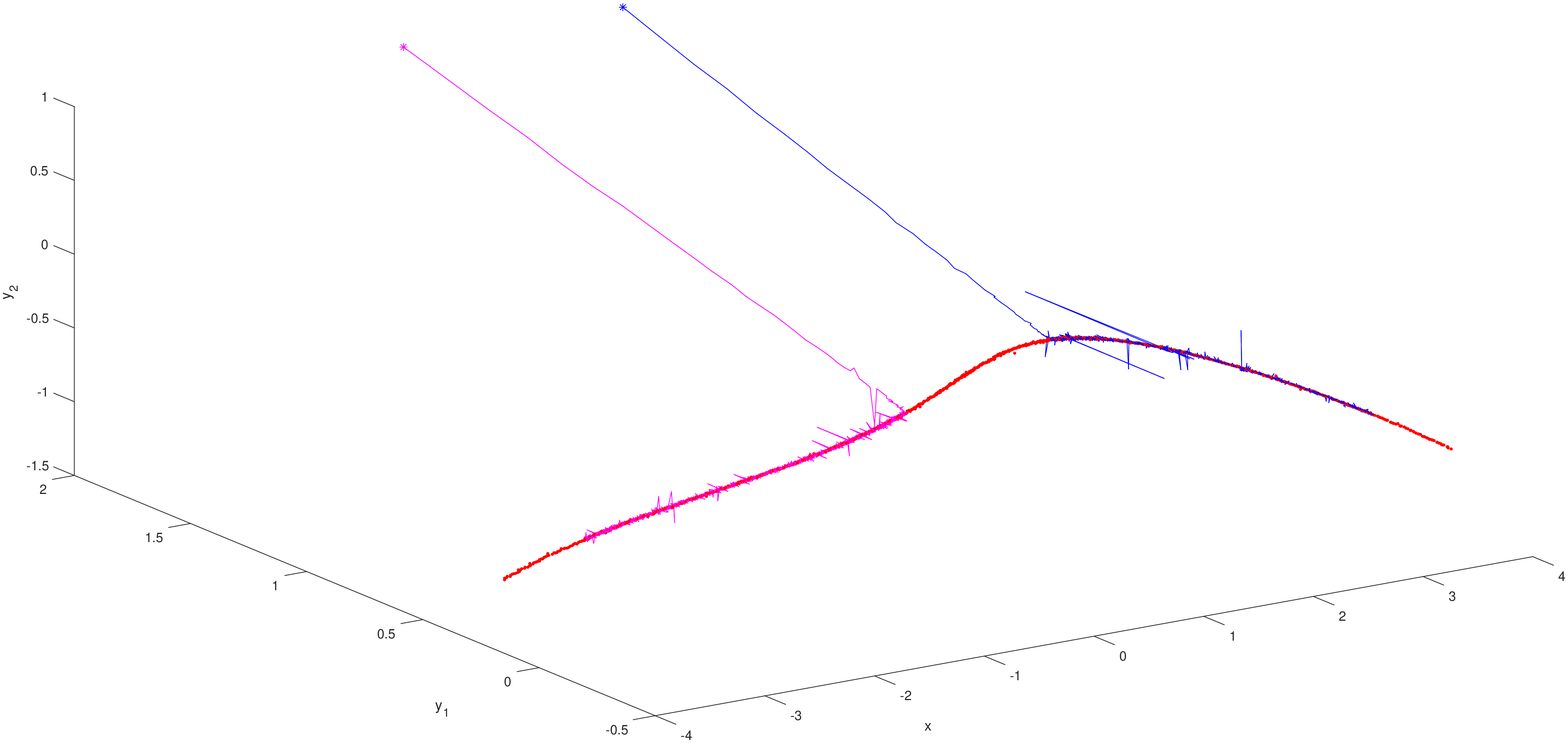}
\includegraphics[width=2.92in]{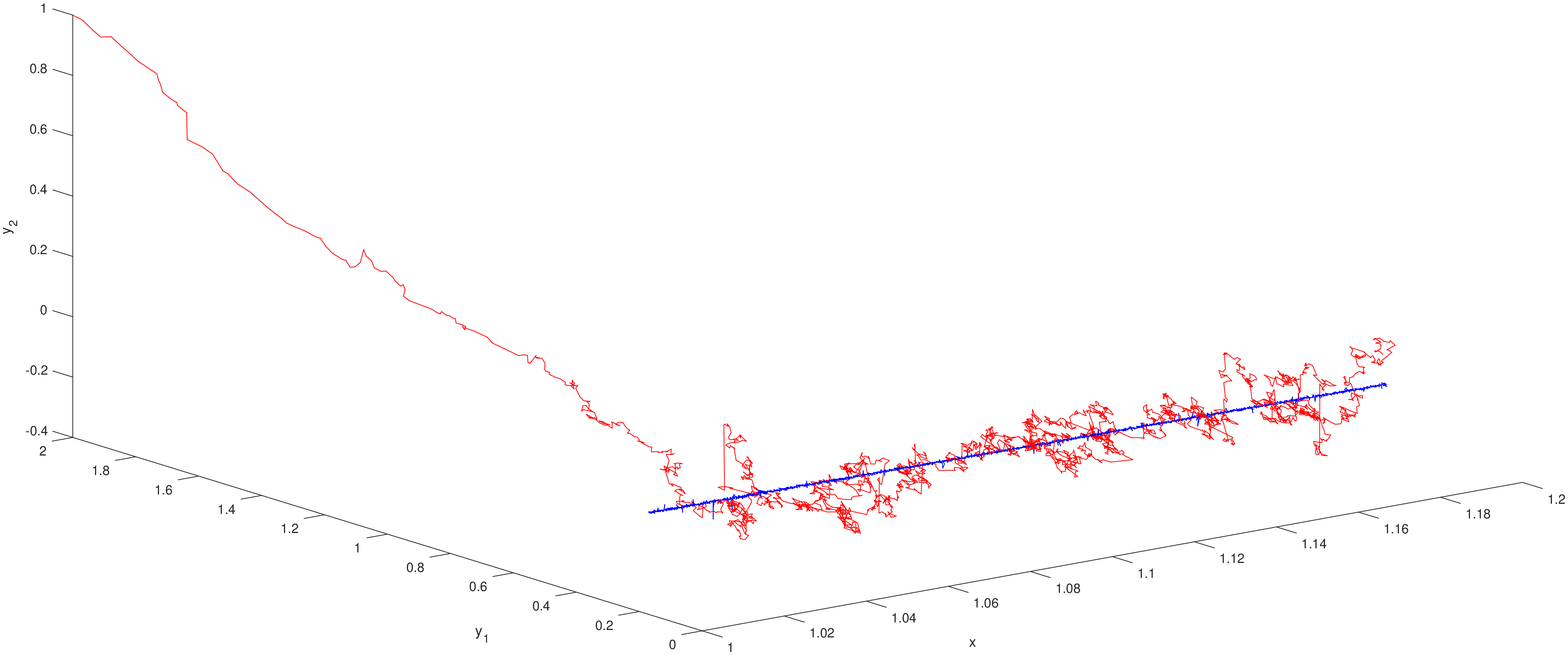}
\caption{(Online color)Two orbits of slow-fast system \eqref{xin3} and its slow manifold expansion $\tilde{h}^{\varepsilon}(\omega,(x_{0},x^{'}_{0}))$; and orbits for system \eqref{xin3} and slow manifold reduced system \eqref{rd3}: $\varepsilon=0.01$, $\alpha_{1}=1.9$, $\alpha_{2}=1.7$, $\sigma=0.1$ .}
\end{figure}
\end{example}
\textbf{Acknowledgements.} The authors are grateful to Bj{\"o}rn Schmalfu{\ss}, Ren\'e Schilling, Georg Gottwald, Jicheng
Liu and Jinlong Wei for helpful discussions on stochastic differenial equations driven by L\'evy motions.


\end{document}